\documentclass[10pt, reqno]{amsart}

\usepackage{graphicx}
\usepackage{amsmath}
\usepackage{amssymb}
\usepackage{amsfonts}
\usepackage{amsthm}
\usepackage[all]{xy}
\usepackage{enumerate}
\usepackage{subfig}
\usepackage[latin1]{inputenc} 
\usepackage{algpseudocode}
\usepackage{color}
\usepackage{enumerate}
\usepackage{longtable}
\usepackage{array}
\usepackage{float}
\usepackage{mathdots}
\usepackage{yhmath}
\usepackage{MnSymbol}
\usepackage{footnote}
\restylefloat{table}
\usepackage{stmaryrd}
\usepackage{hyperref}
\usepackage{fullpage}

\newtheorem{theorem}{Theorem}[section]
\newtheorem{definition}[theorem]{Definition}

\newtheorem{proposition}[theorem]{Proposition}

\newtheorem{lemma}[theorem]{Lemma}
\newtheorem{example}[theorem]{Example}
\newtheorem{remark}[theorem]{Remark}

\newcommand{\Aut}{\operatorname{Aut}}

\title{Twists of non-hyperelliptic curves of genus $3$}

\author{Elisa Lorenzo Garc\'ia}

 \makeatletter
\let\@wraptoccontribs\wraptoccontribs
\makeatother

\thanks{}

\begin{document}


\begin{abstract} In this paper we explicitly compute equations for the twists of all the smooth plane quartic curves defined over a number field $k$. Since the plane quartic curves are non-hyperelliptic curves of genus $3$ we can apply the method developed in \cite{Lor14}. The starting point is a classification due to Henn of the plane quartic curves with non-trivial automorphism group up to $\mathbb{C}-$isomorphism. 
\end{abstract}


\maketitle

\section{Introduction}\label{Intro}

The twists of curves of genus less or equal than $2$ over number fields are well-known.
While the genus $0$ and $1$ cases date back from a long time ago, see \cite{Sil}, the genus $2$
case is due to the work of Cardona and Quer \cite{Cart}, \cite{Carp}. All the 
genus $0$, $1$ or $2$ curves are hyperelliptic (at least
in the sense that they are not non-hyperelliptic, since genus $0$ and $1$ curves are not usually called
hyperelliptic). However, for genus greater than $2$, 
the generic curve is non-hyperelliptic. We devote
the present paper to explicitly compute the twists of the genus $3$ non-hyperelliptic curves defined over a number field $k$.

Since the image of the canonical morphism of a non-hyperelliptic curve of genus $3$ is a degree $2g-2=4$ curve into $\mathbb{P}^{2}$, it is a plane quartic curve. Moreover, if the non-hyperelliptic genus $3$ curve is defined over $k$, then we can take a $k-$isomorphic plane quartic curve defined over $k$ since its canonical model is also defined over $k$. Conversely, all non-singular plane quartic curves are non-hyperelliptic genus $3$ curves. So, from now on, we will indistinctly speak about non-singular plane quartic curves or smooth non-hyperelliptic genus $3$ curves.

Since there is a one to one correspondence between the set of twists of a smooth curve and the first cohomology set
$$
\operatorname{Twist}_k(C)\leftrightarrow\operatorname{H}^{1}_{et}(\text{Gal}(\bar{k}/k),\text{Aut}(C)),
$$
see \cite{Lor14}, \cite{Sil},  if a curve has trivial automorphism group, then the set of twists is also trivial. Therefore, in order to compute twists, we can restrict our attention to plane quartic curves with non-trivial automorphism group. Henn Classification provides a classification up to $\mathbb{C}$-isomorphism of such curves.

To compute the twists, we use the method described in \cite{Lor14} and sometimes, when the twists are easy to compute, we use weaker, but more convenient results.

In a forthcoming paper, see \cite{FLS}, this classification of twists will be useful for computing new Sato-Tate groups and distributions coming from twists of the Fermat and Klein quartic.

\subsection{Outline} The structure of this paper is as follows. In Section \ref{clasif}, we present Henn classification of plane quartic curves with non-trivial automorphism group together with slightly modifications in order to exhibit what we called representative families of curves for the $12$ cases that appear. We recall the method developed in \cite{Lor14} in Section \ref{lemmas}, and we also present some other results that will facilitate the computation of the twists. Computing the twists of the Fermat quartic and the Klein quartic are the two more difficult cases. We devote Section \ref{SecFermat} to compute the twists of the Fermat quartic by using the method in \cite{Lor14}. In Section \ref{Otroscasos}, we use the knowledge of its twists plus the results in Section \ref{lemmas} to compute the twists of the quartics in the first $10$ cases of Henn classification. Finally, in Section \ref{SecKlein}, we enumerate the twists of the Klein quartic.

We include as an appendix some tables with generators of the automorphism groups of the different families in Henn classification and some intermediate data needed to compute the twists. 

\subsection{Notation}
We now fix some notation and conventions that will be valid through the paper. For any field $F$, we denote by $\bar{F}$ an algebraic closure of $F$, and by $G_F$ the absolute Galois group $\operatorname{Gal}(\bar{F}/F)$. We recurrently consider the action of $G_F$ on several sets, and this action will be in general denoted by left exponentiation. For a field $F$, let $\operatorname{GL}_n(F)$ (resp. $\operatorname{PGL}_n(F)$) denotes the ring of $n$ by $n$ invertible matrices with coefficients in F (resp. that are projective). 

By $k$ we always mean a number field. All field extensions of $k$ are considered to be contained in a fixed algebraic closure $\bar{k}$. We write $\zeta_n$ for a primitive $n-$th root of unity in $\bar{k}$. We denote by $\mathcal{O}_k$ the ring of integers of $k$. We will usually denote elements in $k$ by Latin letters $a,b,c,...,m,n,...$ and elements in $\bar{k}$ by Greek letters $\alpha,\beta,\gamma,...$

All the curves $C/k$ are considered to be projective, smooth and geometrically irreducible. We will denote by $\operatorname{Aut}(C)$ the automorphisms group of $C$ over $\bar{k}$, and by $K/k$ the minimal extension such that all the automorphism of $C$ can be defined over. By $\Omega^1(C)$, we denote its $k-$vector space of regular differentials.

When we work with groups we usually use the SmallGroup Library-GAP notation, \cite{GAP}. Where the group $\operatorname{GAP}(N,r)$ or $<N,r>$ denotes the group of order $N$ that appears in the $r-$th position in such library. Given a group $G$, $\text{ID}(G)$ means its GAP notation. Cyclic groups, symmetric groups and dihedral groups with $n$, $n!$, $2n$ elements are written as $\text{C}_n$, $S_n$, $D_n$. By $V_4$, we mean the Klein group isomorphic to two copies of $\text{C}_2$.

\subsection{Aknowledgments} The author would like to thank Joan-Carles Lario for bringing this problem to her attention and Francesc Fit\'e for carefully reading the manuscript, double checking the computations and useful comments and suggestions.

\section{Henn classification of plane quartic curves with non-trivial automorphism group}\label{clasif}

Let $C_1$ be a projective, non-singular, irreducible curve, and $C_2\in\operatorname{Twist}_k(C_1)$. We have the equality of sets $\operatorname{Twist}_k(C_1)=\operatorname{Twist}_k(C_2)$. Hence, for our purposes, it is enough to compute $\operatorname{Twist}_k(C)$ for $C$ a representative for each class of non-hyperelliptic genus $3$ curves defined over $k$ up to $\bar{k}$-isomorphism. Moreover, as we already mentioned, if $\text{Aut}(C_1)$ is trivial, the set $\operatorname{Twist}_k(C_1)$ is also trivial. Henn classification, see \cite{Hen}, \cite{Ver}, classifies plane quartic curves over $\mathbb{C}$ with non-trivial automorphism groups up to $\mathbb{C}$-isomorphism. There are $12$ possibilities for these automorphism groups. Henn classification shows $12$ different families that parametrize all such cases.

Unfortunately, the stratification of the coarse moduli space of curves of genus $3$, $\text{M}_3$, provided by these families, is not good enough in the sense that they do not represent each geometric point of the moduli space over a non algebraically closed field. In other words, the problem is that given a plane quartic curve $C$ with non-trivial automorphism group and defined over a number field $k$, its representative in Henn classification is not necessarily defined also over $k$. In \cite{LRRS} and \cite{Tesis}, variations of Henn classification are given in order to get families with this property.
In \cite[Section 2]{Tesis}, it is explained how to compute, given a non-hyperelliptic curve of genus $3$, a representative in the modified Henn classification.

\subsection{Henn classification}
\label{21}

In $1976$ Henn gives the next classification up to $\mathbb{C}-$ isomorphism of the non-singular plane quartic curves with non-trivial automorphism group:

$ $
\begin{center}
	\begin{tabular}{|c|c|c|c|}
		\hline 
		Case & Model & $\operatorname{Aut}\left(C\right)$ & Parameter restrictions \tabularnewline 
		\hline 
		I & $x^{4}+x^{2}F_1\left(y,z\right)+F_2\left(y,z\right)=0$ & $\text{C}_{2}$ &$F_1\left(y,z\right)\neq0$,\,not\,\,below\tabularnewline
		\hline 
		II & $x^{4}+y^{4}+z^{4}+ax^{2}y^{2}+by^{2}z^{2}+cz^{2}x^{2}=0$ & $V_{4}$&$a\neq\pm b\neq c\neq\pm a$\tabularnewline
		\hline 
		III & $z^{3}y+x\left(x-y\right)\left(x-ay\right)\left(x-by\right)=0$ & $\text{C}_{3}$&not\,\,below\tabularnewline
		\hline 
		IV & $x^{3}z+y^{3}z+x^{2}y^{2}+axyz^{2}+bz^{4}=0$ & $S_{3}$&$a\neq b$ and $ab\neq0$\tabularnewline
		\hline 
		V & $x^{4}+y^{4}+z^{4}+ax^{2}y^{2}+bxyz^{2}=0$ & $D_{4}$&$b\neq0,\:\pm\frac{2a}{\sqrt{1-a}}$\tabularnewline
		\hline 
		VI & $z^{3}y+x^{4}+ax^{2}y^{2}+y^{4}=0$ & $\text{C}_{6}$& $a\neq0$\tabularnewline
		\hline 
		VII & $x^{4}+y^{4}+z^{4}+ax^{2}y^{2}=0$ & $\operatorname{GAP}\left(16,13\right)$&$\pm a\neq0,\,2,\,6,\,2\sqrt{-3}$\tabularnewline
		\hline 
		VIII & $x^{4}+y^{4}+z^{4}+a\left(x^{2}y^{2}+y^{2}z^{2}+z^{2}x^{2}\right)=0$ & $S_{4}$&$a\neq0,\,\frac{-1\pm\sqrt{-7}}{2}$\tabularnewline
		\hline 
		IX & $x^{4}+xy^{3}+yz^{3}=0$ & $\text{C}_{9}$&-\tabularnewline
		\hline 
		X & $x^{4}+y^{4}+xz^{3}=0$ & $\operatorname{GAP}\left(48,33\right)$&-\tabularnewline
		\hline 
		XI & $x^{4}+y^{4}+z^{4}=0$ & $\operatorname{GAP}\left(96,64\right)$&-\tabularnewline
		\hline 
		XII & $x^{3}y+y^{3}z+z^{3}x=0$ & $\operatorname{PSL}_2\left(\mathbb{F}_7 \right)$&-\tabularnewline
		\hline 
	\end{tabular}
\end{center}

$ $

$ $

Where ``not below'' means not $\mathbb{C}-$isomorphic to any model below. \footnote{We show an example of a plane quartic curve defined over $\mathbb{Q}$ such that its representative in Henn classification is not defined over $\mathbb{Q}$: the quartic curve $5x^4+y^4+z^4+x^2y^2=0$ in case VII has as representative the curve with parameter $a=1/\sqrt{5}$. Moreover, notice that the representative does not have to be unique, in the former case we can also take $a=-1/\sqrt{5}$.}

In table \ref{aut} in the Appendix, generators of each of these automorphism groups are given.

We summarize all the previous results in the following proposition. Moreover, in table \ref{autM} in the Appendix, generators of each of these automorphism groups are given.

\begin{proposition}[\cite{Tesis},\cite{LRRS}] The following families make Henn classification a representative classification of plane quartic curves in the sense that it represents each geometric point of the moduli space over non algebraically closed fields. 
	$ $
	
	\begin{table}[H]
		\label{MHenn}
		\begin{center}
			\begin{tabular}{|c|c|c|c|}
				\hline 
				Case & Model & $\operatorname{Aut}\left(C\right)$ & Parameter restrictions \tabularnewline 
				\hline 
				II & $F(x+\alpha y+\alpha^2z,\,x+\beta y+\beta^2z,\,x+\gamma y+\gamma^2z)=0$  &$V_4$&$\alpha\neq \beta\neq \gamma\neq \alpha$\tabularnewline
				\hline 
				\multicolumn{4}{|c|}{where $F(X,Y,Z)=\alpha X^4+\beta Y^4+\gamma Z^4+X^2Y^2+Y^2Z^2+Z^2X^2$ and} \tabularnewline
				\multicolumn{4}{|c|}{$\alpha,\beta,\gamma$ are the three roots of a degree $3$ polynomial with coefficients in $k$}\tabularnewline
				\hline
				III & $z^{3}y+P(x,y)=0$ & $C_{3}$&not below\tabularnewline
				\hline 
				V & $ax^{4}+y^{4}+z^{4}+bx^{2}y^{2}+xyz^{2}=0$ & $D_{4}$&$b\neq0,\,a\neq4b^2(2b+1)^2$\tabularnewline
				\hline 
				VI & $z^{3}y+ax^{4}+x^{2}y^{2}+y^{4}=0$ & $C_{6}$& -\tabularnewline
				\hline 
				VII & $ax^{4}+y^{4}+z^{4}+x^{2}y^{2}=0$ & $\operatorname{GAP}\left(16,13\right)$&$\pm a\neq1/4,\,1/36,\,1/-12$\tabularnewline
				\hline 
			\end{tabular}
		\end{center}
	\end{table}
	
\end{proposition}

\section{The method for computing twists}\label{lemmas}

In this Section, we recall the method developed in \cite{Lor14} for computing twists of non-hyperelliptic curves. We also state some other weaker results that are sometimes more convenient to use for computing the twists in the easiest cases instead of the general method.

Let $C$ be a non-hyperelliptic curve of genus $3$ defined over a number field $k$. Let $K/k$ be the minimal extension over which all its automorphisms can be defined. The extension $K/k$ is finite and Galois, see \cite[Section $2$]{Lor14}, and $\operatorname{Gal}(K/k)$ acts naturally on $\operatorname{Aut}(C)$. We define the twisting group $\Gamma:=\operatorname{Aut}(C)\rtimes\operatorname{Gal}(K/k)$. Let us consider a twist $\phi:\,C'\rightarrow C$ and let us denote by $L$ its field of definition. Clearly, $L$ is a finite extension of $K$ and the extension $L/k$ is Galois. The isomorphism $\phi$ defines a cocycle in $\xi\in\operatorname{H}^{1}(G_k,\operatorname{Aut}(C))$ by  $\xi_{\sigma}=\phi\cdot^{\sigma}\phi^{-1}$, and it also defines a proper solution to the Galois embedding problem (see Theorem $2.2$ in \cite{Lor14})

\begin{align}
\xymatrix{
	&            &                   & G_k \ar@{-->>}_-{\Psi}[dl] \ar@{->>}[d]   &    \\
	1 \ar[r] &  \operatorname{Aut}(C)\,\ar@{^{(}->}[r]   & \Gamma \ar[r]_-{\pi}  & \operatorname{Gal}(K/k)  \ar[r] &  1
}. \label{gep}
\end{align}

Let us put $G:=\operatorname{Gal}(L/k)$ and $H:=\operatorname{Gal}(L/K)$.

\subsection{The General Method}

We recall here the main steps to compute the twists of a non-hyperelliptic genus $3$ curve according to the method developed in \cite{Lor14}.

\begin{itemize}
	\item Step 1:
	Compute the canonical model: since we work with genus $3$ curves, the canonical model is the plane model, so the one that appears in the modified Henn classification.
	
	\item Step 2:
	Compute the groups $(G,H)$, that is, the candidates to Galois groups of the extensions $L/k$ and $L/K$ where $L$ is the field of definition of a twist, and find the proper solutions to the corresponding Galois embedding problems.
	
	\item Step 3:
	Find equations for the twists.
\end{itemize}

\subsection{Some useful results}\label{usefullemmas}

Computing the twists of plane quartic curves is sometimes easier by other methods rather than by the general one explained before. For these particular cases, we use the remarks and lemmas below. 

\begin{remark} \label{meq} Two twists $\phi_i:\,C_i\rightarrow C$ of a plane quartic curve are equivalent if and only if, considering the isomorphisms $\phi_i$'s as elements in $\operatorname{GL}_{3}\left(\bar{k}\right)$ (via its action on $\Omega^1(C)$), there exists a matrix $M\in\operatorname{GL}_{3}\left(k\right)$ and an automorphism $\alpha\in\operatorname{Aut}(C)$ such that $\alpha\circ\phi_1=\phi_2\circ M$; In other words, they are equivalent if and only if, up to an automorphisms of $C$, the columns of $\phi_1$ are $k-$linear combinations of the columns of $\phi_2$. 
\end{remark}

We state now Lemma $1.3.3$ in \cite{Tesis} for $g=3$.

\begin{lemma}\label{forma} Let $C/k$ be a plane quartic curve, and $\xi\in\operatorname{H}^1(G_k,\operatorname{Aut}(C))$ a cocycle, with splitting field $L$. Assume that the elements in the image $\xi(\operatorname{Gal}(L/k))\subseteq \operatorname{Aut}(C)$, as matrices in $\operatorname{GL}_3(L)$, have the form of block matrices
	\begin{align}
	\left(\begin{array}{cc} A&0\\ 0&a\end{array}\right),\label{F}
	\end{align}
	where $A\in\operatorname{GL}_2(L)$ and $a\in L$. Then, there exists a basis of $\Omega^{1}_{L}(C)^{\operatorname{Gal}(L/k)}_{\xi}$ such that the isomorphism $\phi :\,C'\rightarrow C$ associated to it (see section \ref{clasif}) has the same form as block matrix in (\ref{F}). In the particular case, that all the elements in $\xi(\operatorname{Gal}(L/k))$ are diagonal matrices, we can take a basis of $\Omega^{1}_{L}(C)_{\xi}^{\operatorname{Gal}(L/k)}$ such that $\phi$ is also a diagonal matrix.
\end{lemma}

\begin{remark}\label{inclusion} Let $C_i/k$ for $i\in\{1,2\}$ be two plane quartic curves such that there exists an inclusion of automorphisms groups $\iota:\,\Aut(C_1)\rightarrow\Aut(C_2)$, compatible with the action of $G_k$, that is, such that $^\sigma(\iota(\alpha))=\iota(^{\sigma}\alpha)$ for all $\sigma\in G_k$ and all $\alpha\in\Aut(C_1)$. Then, there is a natural inclusion of the set of cocycles of the first Galois cohomology $\operatorname{Z}^{1}(G_k,\Aut(C_1))\hookrightarrow\operatorname{Z}^{1}(G_k,\Aut(C_2))$. The inclusion does not lift to an inclusion of cohomology sets because we quotient by more elements in the right hand size.
\end{remark}

We show how to use these results to compute twists of plane quartics curves in an example.

\begin{example}\label{examp}
	Let us consider the family VI of plane quartic curves in Henn classification with automorphism group isomorphic to the group $\operatorname{S}_3$.
	$$
	C_{a,b}:\,F_{a,b}(x,y,z)=x^3z+y^3z+x^2y^2+axyz^2+bz^4=0,
	$$
	where $a,b\in k$ are such that $a\neq b$ and $ab\neq 0$. The automorphism group is generated by the matrices:
	$$
	\left(\begin{array}{ccc}\zeta_3&0&0\\0&\zeta_{3}^{2}&0\\0&0&1\end{array}\right),\,\,\left(\begin{array}{ccc}0&1&0\\1&0&0\\0&0&1\end{array}\right).
	$$
	Hence, Lemma \ref{forma} implies that any twist is given by a block matrix as the one given in (\ref{F}). Since the equivalence of twists as in Remark \ref{meq}, we can assume that 
	$$
	\phi=\left(\begin{array}{ccc}\alpha&\alpha\beta&0\\ \gamma&\gamma\delta&0\\0&0&1\end{array}\right):\,C'_{a,b}\rightarrow C_{a,b},
	$$ 
	with $\alpha\beta\gamma\delta\neq 0$.
	So, an equation of the twist is given by
	$$
	C'_{a,b}:\,F_{a,b}(\alpha (x+\beta y),\gamma (x+\delta y),z)=0.
	$$
	which implies 
	\begin{align}
	\alpha\gamma(x+\beta y)(x+\delta y)\in k[x,y]\label{RR1}\\
	\alpha^3(x+\beta y)^3+\gamma^3(x+\delta y)^3\in k[x,y].\label{RR2}
	\end{align}
	
	We deduce from (\ref{RR1}), that $\beta+\delta,\,\beta\delta\in k$, so either $\beta,\,\delta\in k$, or $\beta,\,\delta$ are conjugate numbers in a quadratic extension $k(\sqrt{m})$, where $m\in k^*/k^{*2}$. In the first case, we can take an equivalent twist, see remark \ref{meq}, given by a diagonal matrix, and then we find the twist $\phi=\operatorname{diag}(\sqrt[3]{n},\sqrt[3]{n^2},1)$. In the second case, we can assume
	$$
	\phi = \left( \begin{array}{ccc} \alpha & \alpha \sqrt{m}& 0\\ \gamma & -\gamma\sqrt{m} &0\\0&0&1\end{array}\right).
	$$
	
	Hence, from (\ref{RR1}) and (\ref{RR2}), we get $\alpha\gamma, \alpha^3+\gamma^3,(\alpha^3-\gamma^3)\sqrt{m}\in k$.
	
	We conclude $\alpha=\sqrt[3]{a_1+a_2\sqrt{m}}$ and $\beta=\sqrt[3]{a_1-a_2\sqrt{m}}$ for some $a_1,\,a_2\in k$ such that $a_{1}^{2}-ma_{2}^{2}=q^3$ for some $q\in k$. An equation for such twist is
	$$
	C'_{a,b}:\,2a_1(x^3z+3mxy^2z)+2a_2m(3x^2yz+my^3z)+q^2(x^2-my^2)^2+aq(x^2-my^2)z^2+bz^4=0.
	$$
	Notice that for $m=1$, we recover the twists given by diagonal matrices.
	
	Now, by using Remark \ref{meq}, it is easy to check that two such twists are equivalent if and only if $m=m'$ and there exist $b_1,\,b_2$ such that
	$$
	a_1+a_2\sqrt{m}=
	(b_1+b_2\sqrt{m})^3(a'_1\pm a'_2\sqrt{m}).
	$$ 
\end{example}

\section{Twists of the Fermat quartic}\label{SecFermat}

We consider the Fermat quartic $C_{F}:\: x^{4}+y^{4}+z^{4}=0$ defined over a number field $k$.
The automorphism
group $\mathrm{Aut}\left(C_{F}\right)$ is isomorphic to $<96,64>$
in GAP notation \cite{GAP}, and as subgroup of $\operatorname{PGL}_{3}\left(\bar{k}\right)$
it is generated by the matrices:
\[
s=\left(\begin{array}{ccc}
0 & 0 & 1\\
1 & 0 & 0\\
0 & 1 & 0\end{array}\right),\,\quad t=\left(\begin{array}{ccc}
0 & -1 & 0\\
1 & 0 & 0\\
0 & 0 & 1\end{array}\right),\, \quad u=\left(\begin{array}{ccc}
i & 0 & 0\\
0 & 1 & 0\\
0 & 0 & 1\end{array}\right)\]

Let us firstly suppose that $i\notin k$, then $K=k\left(i\right)$,
$\mathrm{Gal}\left(K/k\right)\simeq\mathbb{Z}/2\mathbb{Z}$ and $\Gamma:=\mathrm{Aut}\left(C_F\right)\rtimes\mathrm{Gal}\left(K/k\right)\simeq$ $<192,956>$.
Then, it is easily checked with Magma \cite{magma}, see for instance the code \cite[Appendix]{Tesis}, that the possible
pairs $\left(G,H\right)$ in Step $2$ in the method in \cite{Lor14}, are the ones given in the tables \ref{GHdiagonal}, \ref{GHalmost} and \ref{GHnondiag} in the Appendix.
We have divided these pairs into three types: diagonal, almost
diagonal and non-diagonal. The diagonal type corresponds
to the cases in which all the elements in $ G \subseteq \operatorname{Aut}(C_F) \rtimes \operatorname{Gal}(K/k)$ have
a diagonal matrix as first component.
The second type will be called almost-diagonal and it
corresponds to the cases where $ G$ is a $2$-group not
included in the diagonal cases of the former type. Finally, the third type will cover the pairs $( G, H)$ of the remaining cases.

For the first two types of twists we will use the techniques and the results of Subsection \ref{usefullemmas} as in Subsection \ref{examp}.

\begin{theorem}[Diagonal twists]\label{D}  The diagonal twists have fields of definition $L=k( i, \sqrt[4]{a}, \sqrt[4]{b})$, and are defined by equations:
	$$
	ax^{4}+by^{4}+z^{4}=0.
	$$
	Two of these twists, $ax^4+by^4+z^4=0$ and $a'x^4+b'y^4+z^4=0$ are equivalent if and only if  there exists $m \in k$ such that the sets $\left\{ a,b,1 \right\}$ and $\left\{ ma', mb', m \right\}$ are congruent modulo $k^{*4}$.
\end{theorem}

\begin{theorem}[Almost diagonal twists]\label{AD} The almost diagonal twists have fields of definition of the form $L=k( i, \sqrt[4]{a+b\sqrt{m}}, \sqrt[4]{a-b \sqrt{m}}, \sqrt{m})$, and are defined by equations:
	$$
	2ax^{4}+8bmx^{3}y+12max^{2}y^{2}+8bm^{2}xy^{3}+2am^{2}y^{4}+z^{4}=0.
	$$
	Two of these twists are equivalents if and only if their fields of definition have the same quadratic subextension $k(\sqrt{m})/k$ and the columns of the matrix associated to one isomorphism $\phi:\, C'_F\to C_F$ are $k-$rational linear combination of the columns of the matrix associated to the other isomorphism $\phi':\,C''_F\to C_F$, see Remark \ref{meq}. That is, if there exist $c,d \in k$ such that
	$$
	(a+b\sqrt{m})=(c+d\sqrt{m})^{4}(a'\pm b'\sqrt{m}).
	$$ 
\end{theorem}

For the non-diagonal twists, we use the method developed in \cite{Lor14}.

\begin{proposition}[Non-diagonal twists]\label{prop} The embedding problems corresponding to the nine pairs $(G,H)$ for non-diagonal twists have a solution. The corresponding fields of definition are $L=k(i, \sqrt[n]{\alpha},\sqrt[n]{\beta},\sqrt[n]{\gamma})$, where $\alpha$, $\beta$, $\gamma$ are the three roots of an irreducible polynomial of degree $3$ with coefficients in $k$ with $\sqrt[n]{\alpha\beta\gamma}\in k$. The different cases are:
	
	\begin{center}
		\begin{tabular}{|c|c|c|c|c|c|c|c|c|c|}
			\hline
			\multicolumn{10}{|c|}{Type III: Non-diagonal twists}\tabularnewline
			\hline
			\hline
			& $ \operatorname{ID}( G) $ \rule[0.4cm]{0cm}{0cm} & $ \operatorname{ID}(H)$ & $\triangle$ mod $k^{*2}$& $n$&
			& $ \operatorname{ID}( G) $ \rule[0.4cm]{0cm}{0cm} & $ \operatorname{ID}(H)$ & $\triangle$ mod $k^{*2}$& $n$\tabularnewline
			\hline
			$1$ & $<6,1>$ & $<3,1>$ & $-1$ & $1$&
			$6$ & $<48,48>$ & $<24,12>$ & $\neq\pm1$ & $2$\tabularnewline
			\hline
			$2$ & $<6,2>$ & $<3,1>$ & $1$ & $1$&
			$7$ & $<96,64>$ & $<48,3>$ & $-1$ & $4$\tabularnewline
			\hline
			$3$ & $<12,4>$ & $<6,1>$ & $\neq\pm1$ & $1$&
			$8$ & $<96,72>$ & $<48,3>$ & $1$ & $4$\tabularnewline
			\hline
			$4$ & $<24,12>$ & $<12,3>$ & $-1$ & $2$&
			$9$ & $<192,956>$ & $<96,64>$ & $\neq\pm1$ & $4$\tabularnewline
			\hline
			$5$ & $<24,13>$ & $<12,3>$ & $1$ & $2$&\multicolumn{5}{|c|}{}\tabularnewline
			\hline
		\end{tabular}
	\end{center}
	
	Here, $\triangle$ denotes the absolute discriminant of the extension $k(\alpha,\beta,\gamma)/k$.
\end{proposition}

\noindent{\it Proof.}  The solutions associated to the first three pairs are well-known, see for example \cite{Gam}.

For the sixth pair, since $\operatorname{Gal}(L/K)$ is isomorphic to $S_{4}$ and $\operatorname{Gal}(L/k)$ is isomorphic to $S_{4}\times\mathbb{Z}_{2}$, we conclude that $L=k_{f}(i)$ where $k_{f}$ is the splitting field of an irreducible monic degree $4$ polynomial $f(x)=x^{4}+a_{2}x^{2}+a_{1}x+a_{0}\in k\left[x\right]$, such that the splitting field of its cubic resolvent $g(x)=x^{3}+2a_{2}x^{2}+(a_{2}^{2}-4a_{0})x-a_{1}^{2} $ has Galois group isomorphic to $S_{3}$. Let $r_{0}$, $r_{1}$, $r_{2}$ and $r_{3}$ be the four roots of $f$, and let us define
$$
s_{1}=\frac{1}{2}(r_{0}-r_{1}+r_{2}-r_{3}),\,\,\,
s_{2}=\frac{1}{2}(r_{0}+r_{1}-r_{2}-r_{3}),\,\,\,s_{3}=\frac{1}{2}(r_{0}-r_{1}-r_{2}+r_{3}).
$$
Then the roots of $g(x)=0$ are $s_{1}^{2}$, $s_{2}^{2}$ and $s_{3}^{2}$ and $g(x^{2})=(x^{2}-s_{1}^{2})(x^{2}-s_{2}^{2})(x^{2}-s_{3}^{2})$ is also irreducible over $k$ and $L=k(s_{1},s_{2},s_{3})$. We set $\alpha=s_{1}^{2}$, $\beta=s_{2}^{2}$ and $\gamma=s_{3}^{2}$, then $L=k(\sqrt{\alpha}, \sqrt{\beta}, \sqrt{\gamma})$ with $\sqrt[]{\alpha\beta\gamma}\in k$. 

Similar arguments yield the solutions to the Galois embedding problems for the pairs $4$ and $5$.

For the ninth case one considers the Galois extension $M=L^{G_0}$ over $k$ given by the normal subgroup $G_0=\left\langle u\rtimes1,t^{3}ut\rtimes1,1\rtimes\tau\right\rangle\lhd\text{Gal}(L/k)$. It has Galois group isomorphic to $S_{3}$ and its quadratic subextension is different from $k(i)$. Consider now the Galois extension $M_{1}=L^{G_1}$ given by the subgroup $G_1=\left\langle u\rtimes1\right\rangle$. Then $\operatorname{Gal}(M_{1}/M(i))\simeq\mathbb{Z}_4$ and $\operatorname{Gal}(M_{1}/M)\simeq D_{4}$. Hence, by applying proposition $4.1$ in \cite{Lor14}, $M_{1}=M(i,\sqrt[4]{\alpha})$ with $\alpha\in M$. Since $M_1/k(\alpha )$ is a normal extension we conclude $\operatorname{Gal}(k(\alpha)/k)\simeq \mathbb{Z}_3$.  Idem with $M_{2}$ given by the subgroup $G_2\left\langle t^{3}ut\rtimes1\right\rangle$, we get $M_{2}=L^{G_2}=M(i,\sqrt[4]{\beta})$ with $\beta\in M$ and $\operatorname{Gal} (k(\beta)/k)\simeq \mathbb{Z}_3$. Also $k(\alpha)\neq k(\beta)$. Since $G_1\cap G_2=\{1\}$, we have $L=M_{1}M_{2}$. Finally, since $L/k$ is a normal extension and there is no other normal extension over $k(\beta)$ having Galois group isomorphic to $D_4\times\mathbb{Z}_2$, we can take $\beta$ to be a conjugate of $\alpha$. Let $\gamma$ be the third conjugated. If we inspect the action of $\operatorname{Gal}(L/k)$ on $\sqrt[4]{\alpha\beta\gamma}$ we obtain $\sqrt[4]{\alpha\beta\gamma} \in k$, and then the result follows. 

The solutions to the Galois embedding problems coming from pairs $7$ and $8$ follow using the same arguments.
\hfill $\square$

\begin{remark}\label{R3} In the first six cases two of these twists are equivalents if and only if they have the same splitting field. In the cases seventh and ninth, the same field $L$ provides two different twists, because in that case, see formula $4.2$ in \cite{Lor14},
	$$
	n_{(G,H)}:=\mathrm{Aut}_{H}\left(G\right)/Inn_{G}\left(\mathrm{Aut}\left(C_F\right)\rtimes\left\{ 1\right\} \right)=2.
	$$
	In the eighth one, each field $L$ provides four different twists.
	
\end{remark}

\begin{theorem}[Non-diagonal twists]\label{ND} A non-diagonal twist with field of definition $L=k(i, \sqrt[n]{\alpha},\sqrt[n]{\beta},\sqrt[n]{\gamma})$ is defined by the equation:
	$$
	\sum_{\begin{array}{c}j+k+l=4\\j,k,l\geq0\end{array}}\binom{4}{j}\binom{4-j}{k}S_{e_n+k+2l}x^{j}y^{k}z^{l}=0,
	$$
	where $S_{j}=\alpha^{j}+\beta^{j}+\gamma^{j}$, and
	$
	e_n=\small{\begin{cases} 0\text{ if }n=1 \\
		2\text{ if }n=2\\
		1\text{ if }n=4\end{cases}}.
	$
	
	If $n=4$ this splitting field produces more than one twist. For the seventh and ninth cases we have also the twist coming from replacing $\alpha,\beta,\gamma$ with $\alpha^3,\beta^3,\gamma^3$ and for the eighth one, we also have the twists coming from replacing $\alpha,\beta,\gamma$ with $\alpha/\beta,\beta/\gamma,\gamma/\alpha$ and $\alpha^3/\beta^3,\beta^3/\gamma^3,\gamma^3/\alpha^3$.
\end{theorem}

\noindent{\it Proof.} For $n=1,2,4$ we can respectively take the isomorphism $\phi\colon C'_F\to C_F$:
$$
\phi = \left(\begin{array}{ccc}1 & \alpha & \alpha^{2}\\1 & \beta & \beta^{2}\\1 & \gamma & \gamma^{2}\end{array}\right),\,\,\,\phi = \left(\begin{array}{ccc}\sqrt{\alpha} &\alpha\sqrt{\alpha}  &\alpha^{2}\sqrt{\alpha}\\ \sqrt{\beta} & \beta\sqrt{\beta} &\beta^{2} \sqrt{\beta}\\ \sqrt{\gamma} & \gamma\sqrt{\gamma} & \gamma^{2}\sqrt{\gamma}\end{array}\right),\,\,\,
\phi = \left(\begin{array}{ccc}\sqrt[4]{\alpha} &\alpha\sqrt[4]{\alpha}  &\alpha^{2}\sqrt[4]{\alpha}\\ \sqrt[4]{\beta} & \beta\sqrt[4]{\beta} &\beta^{2} \sqrt[4]{\beta}\\ \sqrt[4]{\gamma} & \gamma\sqrt[4]{\gamma} & \gamma^{2}\sqrt[4]{\gamma}\end{array}\right),
$$
and the ones coming from replacing $a,b,c$ with the numbers in the statement of the theorem. The equivalence of twists is a consequence of Remark \ref{R3}.
\hfill $\square$

\begin{remark}Now, if $k=k(i)$, we define $k_0=k\cap \mathbb{R}$, then $\left[k:k_0\right]=2$ and $i\notin k_0$. We obtain for $k$ the pairs $(H,H)$ where $(G,H)$ is pair for $k_0$.  Then we obtain once again the statements of Theorems \ref{D}, \ref{AD} and \ref{ND}.
\end{remark}

\begin{definition}\label{pol} We define the sets $Pol_3^n(k)$ to be the sets separable polynomials of degree $3$ with coefficients in $k$ and whose independent coefficient is in $-1\cdot k^{*n}$. We define the equivalence relation given by: $P(T)\sim P'(T)$ if and only if they have the same splitting field $M$ and their roots satisfy
	$$
	\text{Roots}(P(T))\equiv\text{Roots}(P'(T)) \text{mod} \:  M^{*4}.
	$$
\end{definition}

We can summarize all these results in the following theorem.

\begin{theorem} The set of $k$-isomorphism classes of non-hyperelliptic genus $3$ curves with automorphism group isomorphic to $\langle 96,64 \rangle$ is parametrized by the set $Pol_3^4(k)/\sim$. A representative plane quartic corresponding to  $P(T)=(T-\alpha)(T-\beta)(T-\gamma)$ is given by:
	$$
	\sum_{\begin{array}{c}i+j+k=4\\i,j,k\geq0\end{array}}\binom{4}{i}\binom{4-i}{j}S_{1+j+2k}x^{i}y^{j}z^{k}=0
	$$
	where $S_j=\alpha^j+\beta^j+\gamma^{j}$.
\end{theorem}

\noindent{\it Proof.}
The only non-hyperelliptic genus $3$ curve up to $\mathbb{C}-$isomorphism with automorphism group isomorphic to $\langle 96,64 \rangle$ is the Fermat quartic, see section \ref{21}. Then we only have to parametrize its twists. Then, it is clear that given such a polynomial $P(T)\in Pol_3^4(k)$ with roots $\alpha,\beta,\gamma$ we can attach to it the twist given by the isomorphism
$$
\phi = \left(\begin{array}{ccc}\sqrt[4]{\alpha} &\alpha\sqrt[4]{\alpha}  &\alpha^{2}\sqrt[4]{\alpha}\\ \sqrt[4]{\beta} & \beta\sqrt[4]{\beta} &\beta^{2} \sqrt[4]{\beta}\\ \sqrt[4]{\gamma} & \gamma\sqrt[4]{\gamma} & \gamma^{2}\sqrt[4]{\gamma}\end{array}\right):\,C'\rightarrow C_F
$$
that gives the equation in the statement of the theorem.

\section{Twists of the quartics in cases from I to X}\label{Otroscasos}

We use the results explained in Section \ref{lemmas}. When no more explanations are given, the proof of the propositions is just a straightforward implication from Remark \ref{meq} and Lemma \ref{forma}.

\begin{proposition} Let $C$ be a non-hyperelliptic genus $3$ curve defined over a number field $k$ and with automorphism group isomorphic to $\text{C}_2$. Then, it is isomorphic to a plane quartic defined over $k$ with equation
	$$
	x^4+x^2F_1(y,z)+F_2(y,z)=0.
	$$
	All its twists are in one to one correspondence with the set $k^{*}/k^{*2}$. Given $m\in k^{*}/k^{*2}$ one has the twist
	$$
	m^2x^4+mx^2F_1(y,z)+F_2(y,z)=0.
	$$
\end{proposition}

\begin{proposition} Let $\mathcal{C}$ be a non-hyperelliptic genus $3$ curve defined over a number field $k$ such that $\operatorname{Aut}(\mathcal{C})\simeq V_4$. Then, there exist $\alpha,\beta,\gamma\in\bar{k}$, that are roots of a separable polynomial of degree $3$ with coefficients in $k$, and such that $\mathcal{C}$ is isomorphic to the model $C_{\alpha,\beta,\gamma}$ of the modified Henn classification. Depending on the isomorphism class of the group $\mathcal{G}:=\operatorname{Gal}(k(\alpha,\beta,\gamma)/k)$, we find the following twists:
	
	i) If $\mathcal{G}=\lbrace 1\rbrace$, the twists of $\mathcal{C}$ are in one-to-one correspondence with the pairs $(m,n)\in k^{*}/k^{*2}\times k^{*}/k^{*2}$. For a pair $(m,n)$, we have the twist of equation:
	$$
	\alpha m^2x^4+\beta n^2y^4+\gamma z^4+mnx^2y^2+ny^2z^2+mz^2x^2=0.
	$$
	ii) If $\mathcal{G}\simeq \text{C}_2$, then $k(\alpha,\beta,\gamma)=k(\sqrt{m})$, for some $m\in k^*/ k^{*2}$. Thus, up tu permutation, we get the conditions $\gamma \in k$, and $\alpha,\,\beta\in k(\sqrt{m})/k$ are conjugate numbers. Then, the twists of $\mathcal{C}$ are given by the isomorphisms
	$$
	\phi= \left(\begin{array}{ccc}\sqrt{c+d\sqrt{m}} & \sqrt{m}\sqrt{c+d\sqrt{m}}  &0\\ \sqrt{c-d\sqrt{m}} &-\sqrt{m}\sqrt{c-d\sqrt{m}} &0\\ 0 & 0&1 \end{array}\right):\,C'\rightarrow C_{\alpha,\beta,\gamma}.
	$$ 
	Two such isomorphism provide equivalent twists if and only if there exist $e,\, f\in k$ so that
	$$
	c+d\sqrt{m}=(e+f\sqrt{m})^2(c'+d'\sqrt{m}).
	$$
	iii) If $\mathcal{G}\simeq \text{C}_3$ or $S_3$, then the twists are given by isomorphisms
	$$
	\phi = \left(\begin{array}{ccc}\sqrt{\alpha_1} &\alpha_1\sqrt{\alpha_1}  &\alpha_1^{2}\sqrt{\alpha_1}\\ \sqrt{\alpha_2} & \alpha_2\sqrt{\alpha_2} &\alpha_2^{2} \sqrt{\alpha_2}\\ \sqrt{\alpha_3} & \alpha_3\sqrt{\alpha_3} & \alpha_3^{2}\sqrt{\alpha_3}\end{array}\right):\,C'\rightarrow C_{\alpha,\beta,\gamma},
	$$
	where $\alpha_1\in k(\alpha)$, $\alpha_2\in k(\beta)$ and $\alpha_3\in k(\gamma)$ are conjugate numbers in $k(\alpha,\beta,\gamma)$. Two such twists are equivalent if and only if $\alpha_1,\alpha_2,\alpha_3\equiv \alpha'_1,\alpha'_2,\alpha'_3\,\operatorname{mod}\,k^{*2}$
\end{proposition}

\begin{proof}
	Apply remark \ref{inclusion} to the curves $C_1:\,\alpha x^4+\beta y^4+\gamma z^4+x^2y^2+y^2z^2+z^2x^2=0$ and the diagonal twist of the Fermat quartic, (see section \ref{SecFermat}), $C_2:\alpha x^4+\beta y^4+\gamma z^4=0$. Finally, check which cocycles $\operatorname{Z}^1(\operatorname{G}_k,\operatorname{Aut}(C_2))$ actually define twists of $C_1$.
\end{proof}

\noindent \textbf{Remark.} \textit{Finding different twists for the quartics $C_{\alpha,\beta,\gamma}$ depending on the Galois group $\mathcal{G}$ is completely natural. Since this group determines the action of $G_k$ on $\operatorname{Aut}(C)$. This will also happen for cases V and VII in the modified Henn classification.}

\begin{proposition} Let $C$ be a non-hyperelliptic genus $3$ curve defined over a number field $k$ and with automorphism group isomorphic to $\text{C}_3$. This curve is isomorphic to a plane quartic curve, defined over $k$, of equation
	$$
	z^3y+P(x,y)=0.
	$$
	Furthermore, all its twists are in one to one correspondence with the set $k^{*}/k^{*3}$. Given $m\in k^{*}/k^{*3}$ one has the twist
	$$
	mz^3y+P(x,y)=0.
	$$
\end{proposition}

The following proposition is example \ref{examp}.

\begin{proposition}  Let $C$ be a non-hyperelliptic genus $3$ curve defined over a number field $k$ such that $\operatorname{Aut}(C)\simeq S_3$. Then, there exist $a,b\in k$, such that $C$ is isomorphic to  $C_{a,b}$ in case IV of Henn classification. All its twists are parametrized by
	$$
	2a_1(x^3z+3mxy^2z)+2a_2m(3x^2yz+my^3z)+q^2(x^2-my^2)^2+aq(x^2-my^2)z^2+bz^4=0,
	$$
	where $a_1,\,a_2,\,m,\,q\in k$ satisfy the relation $a_{1}^{2}-ma_{2}^{2}=q^3$. Two such twists are equivalent if and only if they have the same parameter $m\in k^{*}/k^{*2}$ and there exist $b_1,\,b_2\in k$ such that
	$$
	a_1+a_2\sqrt{m}=(b_1+b_2\sqrt{m})^3(a'_1\pm a'_2\sqrt{m}).
	$$
	
\end{proposition}

\begin{proposition}
	Let $C$ be a non-hyperelliptic genus $3$ curve defined over a number field $k$ such that $\operatorname{Aut}(C)\simeq D_4$. Then, there exist $a,b\in k$, such that $C$ is isomorphic to $C_{a,b}$ in case V of the modified Henn classification. All its twists are parametrized by
	$$
	2cx^4+8dmx^3y+12cmx^2y^2+8dm^2xy^3+2cm^2y^4+bq^2(x^2-my^2)^2+q(x^2-my^2)z^2+z^4=0,
	$$
	where $m\in k^{*}/k^{*2}$ and $c,\,d,\,q\in k$ satisfy  $c^2-d^2m=q^4a$. Two such twists are equivalent if and only if they have the same parameter $m\in k^{*}/k^{*2}$ and there exist $e,\,f\in k$ such that
	$$
	c+d\sqrt{m}=(e+f\sqrt{m})^4(c'\pm d'\sqrt{m}).
	$$
\end{proposition}

\begin{proof}
	Apply Remark \ref{inclusion} to the curves $C_1:\,ax^4+y^4+z^4+bx^2y^2+xyz^2=0$ and the diagonal twist of the Fermat quartic, see Section \ref{SecFermat}, $C_2:ax^4+y^4+z^4=0$. Finally, check which cocycles $\operatorname{Z}^1(\operatorname{G}_k,\operatorname{Aut}(C_2))$ define twists of $C_1$.
\end{proof}

\begin{proposition} The set of $k$-isomorphism classes of non-hyperelliptic genus $3$ curves with automorphism group isomorphic to the cyclic group $\text{C}_6$ is parametrized by the set $k^*\times k^{*}/k^{*2}\times k^{*}/k^{*3}$. Given an element $(a,m,n)\in k^*\times k^{*}/k^{*2}\times k^{*}/k^{*3}$, we have the curve
	$$nz^3y+am^2x^4+mx^2y^2+y^4=0.$$
	Moreover, two such curves are $\bar{k}$-isomorphic if and only if they have the same parameter $a\in k^{*}$.
\end{proposition}

\begin{proposition}
	Let $C$ be a non-hyperelliptic genus $3$ curve defined over a number field $k$ such that $\operatorname{Aut}(C)\simeq \operatorname{GAP}(16,13)$. Then, there exists $a\in k$, such that $C$ is isomorphic to $C_{a}$ in case VII of the modified Henn classification. All its twists are parametrized by
	$$
	2bx^4+8cmx^3y+12bmx^2y^2+8cm^2xy^3+2bm^2y^4+q(x^2-my^2)^2+z^4=0,
	$$
	where $m\in k^{*}/k^{*2}$ and $b,\,c,\,q\in k$ satisfy  $b^2-c^2m=q^4A$. Two such twists are equivalent if and only if they have the same parameter $m\in k^{*}/k^{*2}$ and there exist $d,\,e\in k$ such that
	$$
	b+c\sqrt{m}=(d+e\sqrt{m})^4(b'\pm c'\sqrt{m}).
	$$
\end{proposition}

\begin{proof}
	Apply Remark \ref{inclusion} to the curves $C_1:\,ax^4+y^4+z^4+x^2y^2=0$ and the diagonal twist of the Fermat quartic, see Section \ref{SecFermat}, $C_2:ax^4+y^4+z^4=0$. Finally, check which cocycles $\operatorname{Z}^1(\operatorname{G}_k,\operatorname{Aut}(C_2))$ define twists of $C_1$.
\end{proof}

\begin{proposition} The set of $k$-isomorphism classes of non-hyperelliptic genus $3$ curves with automorphism group isomorphic to the symmetric group $S_4$, is parametrized by the set $k^{*}\times Pol_3^2(k)/\sim$ (see definition \ref{pol}). A plane quartic corresponding to $P(T)=(T-\alpha)(T-\beta)(T-\gamma)\in Pol_3^2(k)$ is given by the isomorphism
	$$
	\phi = \left(\begin{array}{ccc}\sqrt{\alpha} &\alpha\sqrt{\alpha}  &\alpha^{2}\sqrt{\alpha}\\ \sqrt{\beta} & \beta\sqrt{\beta} &\beta^{2} \sqrt{\beta}\\ \sqrt{\gamma} & \gamma\sqrt{\gamma} & \gamma^{2}\sqrt{\gamma}\end{array}\right):\,C_{\alpha,\beta,\gamma}\rightarrow C_a:\,x^4+y^4+z^4+a(x^2y^2+y^2z^2+z^2x^2)=0.
	$$
\end{proposition}

\begin{proof}
	Apply Remark \ref{inclusion} to the curves $C_a$ and the Fermat quartic, $C_F:x^4+y^4+z^4=0$. Finally, check which cocycles $\operatorname{Z}^1(\operatorname{G}_k,\operatorname{Aut}(C_F))$ actually define twists of $C_a$.
\end{proof}

\begin{proposition} The set of $k$-isomorphism classes of non-hyperelliptic genus $3$ curves with automorphism group isomorphic to the cyclic group $\text{C}_9$ is parametrized by the set $ k^{*}/k^{*9}$. Given an element $a\in  k^{*}/k^{*9}$, we have the curve
	$$az^3y+x^4+x^2y^2+y^4=0.$$
	All these curves are $\bar{k}$-isomorphic.
\end{proposition}

\subsection*{Case X}
In this case, there is only one curve $C:\,x^4+y^4+xz^3=0$. Remark \ref{forma} implies that its twists are given by isomorphisms of the form
$$
\phi=\left(\begin{array}{ccc}\alpha & 0 &\alpha\beta\\ 0 &1 &0\\ \gamma & 0&\gamma\delta \end{array}\right):\,C'\rightarrow C.
$$
By looking at the equation of $C'$ that defines this isomorphism, we get the conditions (we set $A=\alpha^4$ and $B=\alpha\gamma^3$),
$$\begin{aligned}
A+B=q_1,\,
4A\beta+B(\beta+3\delta)=q_2,\,
6A\beta^2+B(3\delta^2+3\beta\delta)=q_3\in k\\
4A\beta^3+B(\delta^3+3\beta\delta^2)=q_4,\,
A\beta^4+B\beta\delta^3=q_5\in k\\
\end{aligned}
$$
We can assume $q_1\neq 0$ and $q_2=0$. Hence $B=q_1-A$. From the second equation, we get $A=-\frac{q_1(\beta+3\delta)}{3(\beta-\delta)}$ (notice that $\beta\neq \delta$), and thus $B=\frac{4q_1\beta}{3(\beta-\delta)}$. Substituting in the other equations we get
$$
\beta(2\delta+\beta)=-\frac{q_3}{2q_1},\,
\beta(\beta^2+4\beta\delta+\delta^2)=-\frac{3q_4}{4q_1},\,
\beta^2(2\delta+\beta)^2=-\frac{3q_5}{q_1}.
$$
So, in particular, $q_{3}^2=-12q_1q_5$, $\delta=-\frac{\beta^2+\frac{q_3}{2q_1}}{2\beta}$ and
$$
q_1\beta^4+q_3\beta^2-q_4\beta-\frac{q_{3}^{2}}{12q_1}=0.
$$
So, given $q_1,\,q_3,\,q_4\in k$, we have the twist
$$
C':\, q_1x^4+q_3x^2z^2+q_4xz^3-\frac{q_{3}^{2}}{12q_1}z^4+y^4=0.
$$
After dividing $x$ by a suitable $k$-rational number, the equation by $q_1$, and renaming the parameters, we get the equation
$$
C':\, x^4+ax^2z^2+xz^3-\frac{a^{2}}{12}z^4+by^4=0.
$$
Two such twist are equivalent if and only if $b'=q^4b$ for some $q\in k$, and if $a'=a$, as can be checked using Remark \ref{meq}. \\

We summarize the previous discussion in the following proposition: 

\begin{proposition} The set of $k$-isomorphism classes of non-hyperelliptic genus $3$ curves with automorphism group isomorphic to the group $\operatorname{GAP}(48,33)$ is parametrized by the set $k\times k^{*}/k^{*4}$. Given an element $(a,b)\in k\times k^{*}/k^{*4}$, we have the curve
	$$x^4+ax^2z^2+xz^3-\frac{a^{2}}{12}z^4+by^4=0.$$
	All these curves are $\bar{k}$-isomorphic.
\end{proposition}

\section{Twists of the Klein quartic}\label{SecKlein}

In this section we compute the twists of the Klein quartic 
$$
C_K:\,x^3y+y^3z+z^3x=0,
$$
defined over a number field $k$. Different models of this curve are known (see \cite{Elk} for the first two), and they will be useful in the sequel. Here $\epsilon=\frac{-1+\sqrt{-7}}{2}$ and $\zeta=\zeta_7$.
$$
C_{S_4}:\,x^4+y^4+z^4+3\epsilon(x^2y^2+y^2z^2+z^2x^2)=0,
$$
$$
C_0:\,x^4+y^4+z^4+6(xy^3+yz^3+zx^3)-3(x^2y^2+y^2z^2+z^2x^2)+3xyz(x+y+z)=0,
$$
$$
C_{D_4}:\,\sqrt{-7}\epsilon^2/16(x^4+y^4)+z^4-3\epsilon^2/8x^2y^2+3\epsilon xyz^2=0
$$
Isomorphisms between this curves are given by the matrices 
$$
\phi_1=\left(\begin{array}{ccc}1 & 1+\zeta\epsilon &\zeta^2+\zeta^6\\ 1+\zeta\epsilon &\zeta^2+\zeta^6 &1\\ \zeta^2+\zeta^6 & 1&1+\zeta\epsilon \end{array}\right):\,C_{S_4}\rightarrow C_K,\,\phi_2=\left(\begin{array}{ccc} -\epsilon & 1 &2\epsilon+3\\ 2\epsilon+3 &-\epsilon &1\\ 1 &2\epsilon+3&-\epsilon \end{array}\right):\,C_0\rightarrow C_{S_4}
$$

$$
\phi_3=\left(\begin{array}{ccc}1/2&1/2&0\\-i/2&i/2&0\\0&0&1\end{array}\right):\,C_{S_4}\rightarrow C_{D_4},\,\phi_0=\phi_1\cdot \phi_2:\,C_0\rightarrow C_K
$$

We have $\operatorname{Twist}_k(C_K)=\operatorname{Twist}_k(C_0)$, and the curve $C_0$ has the advantage that its automorphism group $\operatorname{Aut}(C_0)$ is defined over $k(\sqrt{-7})$, instead of $k(\zeta_7)$. So, it will be more convenient computing $\operatorname{Twist}_k(C_0)$. We use the method described in \cite{Lor14}. The group $\operatorname{Aut}(C_0)$ is generated by the matrices $g=\phi_{0}^{-1}\alpha_1\phi_0$, $h=\phi_{0}^{-1}\alpha_2\phi_0$ and $s=\phi_{0}^{-1}\alpha_3\phi_0$.

Let us first assume that $\sqrt{-7}\notin k$, and let $K=k(\sqrt{-7})$ and $\text{Gal}(K/k)=<\tau>$. Then, the options for the pairs $(G,H)$ are (we computed them with the Magma code in \cite[Appendix]{Tesis}): 
\begin{table}[H]
	\caption{Pairs $(G,H)$ for the Klein quartic.}
	\label{pairsKlein}
	\begin{center}
		\begin{tabular}{|c|c|c|c|c|c|}
			\hline
			& $ \operatorname{ID}( G) $ \rule[0.4cm]{0cm}{0cm} & $ \operatorname{ID}(H)$ & $\operatorname{gen}(H)$ & $h$&$n_{(G,H)}$\tabularnewline
			\hline
			$1$ & $<2,1>$ & $<1,1>$ & $ 1 $ & $ 1 $&1\tabularnewline
			\hline
			$2$ & $<4,2>$ & $<2,1>$ & $ s $ & $ 1 $&1\tabularnewline
			\hline
			$3$ & $<6,1>$ & $<3,1>$ & $ h $ & $ s $&1\tabularnewline
			\hline
			$4$ & $<6,2>$ & $<3,1>$ & $ h$ & $ 1 $&1\tabularnewline
			\hline
			$5$ & $<14,1>$ & $<7,1>$ & $ g $ & $ 1 $&2\tabularnewline
			\hline
			$6$ & $<8,1>$ & $<4,1>$ & $ g^2sg^3sg^2 $ & $ g^2sg^5 $&2\tabularnewline
			\hline
			$7$ & $<8,3>$ & $<4,1>$ & $ g^2sg^3sg^2$ & $ 1 $&2\tabularnewline
			\hline
			$8$ & $<12,4>$ & $<6,1>$ & $ h,\,s $ & $ 1 $&1\tabularnewline
			\hline
			$9$ & $<42,1>$ & $<21,1>$ & $ g,\,h $ & $ 1 $&2\tabularnewline
			\hline
			$10$ & $<16,7>$ & $<8,3>$ & $ g^2sg^3sg^2,\,g^2sg^5 $ & $ 1 $&4\tabularnewline
			\hline
			$11$ & $<336,208>$ & $<168,42>$ & $ s,\,g,\,h $ & $ 1 $&2\tabularnewline
			\hline

		\end{tabular}
	\end{center}
	$ $
\end{table}
This means that $ G$ is the group whose elements are $(g,1)$, for $g$ in $H$, together with the elements $(gh,\tau)$, for $g$ in $H$. The number $n_{(G,H)}$ is computed with the formula in Remark \ref{R3} and it is the number of non-equivalent twists defined over a same field $L$ solution to the Galois embedding problem corresponding to a pair $(G,H)$.

\begin{theorem}[Twists of the Klein quartic]\label{TwistsKlein} The following table shows the solutions to the Galois embedding problems corresponding to the $11$ cases in table \ref{pairsKlein} for the twists of the Klein quartic. Explicit isomorphism are also showed.
	
	\vspace{5mm}
	
	\begin{longtable}[H]{|c|c|c|c|}
		\hline
		&&&\\
		& $ L $ \rule[0.4cm]{0cm}{0cm} & $ \phi:\,C'\rightarrow C_0$ & Other twists over $L$\tabularnewline
		&&&\\
		\hline
		&&&\\
		$1$ & $K$ &{\footnotesize $\left(\begin{array}{ccc}1 & 0 &0\\ 0 &1 &0\\ 0 & 0&1 \end{array}\right)$} & $-$ \tabularnewline
		&&&\\
		\hline
		&&&\\
		$2$ & $K(\sqrt{m})$ & {\footnotesize $\left(\begin{array}{ccc}2 & \sqrt{m} &0\\ -3 &0 &\sqrt{m}\\ 1 & -2\sqrt{m}&3\sqrt{m} \end{array}\right)$}& $-$ \tabularnewline
		&&&\\
		\hline
		&&&\\
		$3$ & $\begin{array}{c}K(\alpha,\beta,\gamma)\\ \Delta=-7q^2\end{array}$ & {\footnotesize$\left(\begin{array}{ccc}\sqrt{-7} & -3\alpha+2\beta+\gamma &\alpha\beta-3\beta\gamma+2\gamma\alpha\\ \sqrt{-7} &\alpha-3\beta+2\gamma &2\alpha\beta+\beta\gamma-3\gamma\alpha\\ \sqrt{-7} & 2\alpha+\beta-3\gamma&-3\alpha\beta+2\beta\gamma+\gamma\alpha \end{array}\right)$}& $-$ \tabularnewline
		&&&\\
		\hline
		&&&\\
		$4$ & $\begin{array}{c}K(\alpha,\beta,\gamma)\\ \Delta=q^2\end{array}$  &{\footnotesize $\left(\begin{array}{ccc}1 & -3\alpha+2\beta+\gamma &\alpha\beta-3\beta\gamma+2\gamma\alpha\\ 1 &\alpha-3\beta+2\gamma &2\alpha\beta+\beta\gamma-3\gamma\alpha\\ 1 & 2\alpha+\beta-3\gamma&-3\alpha\beta+2\beta\gamma+\gamma\alpha \end{array}\right)$} & $-$ \tabularnewline
		&&&\\
		\hline
		&&&\\
		$5$ & $\begin{array}{c} \widetilde{K(\sqrt[7]{\beta_1},\sqrt[7]{\beta_2},\sqrt[7]{\beta_3})}\\ K(\beta_1,\beta_2,\beta_3)=k(\zeta_7)\\ \text{see Lemma \ref{caso5}} \end{array}$ &{\footnotesize $\phi_{0}^{-1}\circ \left(\begin{array}{ccc}\sqrt[7]{\beta_1} & \beta_1\sqrt[7]{\beta_1} &\beta_1^2\sqrt[7]{\beta_1}\\ \sqrt[7]{\beta_2} &\beta_2\sqrt[7]{\beta_2} &\beta_2^2\sqrt[7]{\beta_2}\\ \sqrt[7]{\beta_3} &\beta_3\sqrt[7]{\beta_3}&\beta_3^2\sqrt[7]{\beta_3} \end{array}\right)$} & $\beta_1,\beta_2,\beta_3\rightarrow\beta_1^6,\beta_2^6,\beta_3^6$ \tabularnewline
		&&&\\
		\hline
		&&&\\
		$6$ & $\begin{array}{c} K(\sqrt[8]{-7\gamma^2})\\ k(i,\sqrt{2})=k \end{array}$ &{\footnotesize $\phi_{3}^{-1}\circ\text{diag}(\sqrt[8]{-7\gamma^2},\frac{\sqrt[8]{(-7\gamma^2)^7}}{7\gamma},\frac{1}{1+2\sqrt{2}+\sqrt{-7}})$} & $\gamma\rightarrow\gamma^3$ \tabularnewline
		&&&\\
		\hline
		&&&\\
		$7$ & $\begin{array}{c}K(\sqrt{a+b\sqrt{m}})\\ a^2-mb^2=-7mq^2\\{\footnotesize S=\sqrt{a+b\sqrt{m}}}\\{\footnotesize A=\frac{\sqrt{a-b\sqrt{m}}}{\sqrt{-7}\sqrt{a+b\sqrt{m}}}}\end{array}$ &{\footnotesize $\left(\begin{array}{ccc}3 & (5A-1)S &\sqrt{m}(5A+1)S\\ -1 &(3A-3)S &\sqrt{m}(3A+3)S\\ -2 & 6AS&6\sqrt{m}AS \end{array}\right)$}&{\footnotesize $\sqrt{a+b\sqrt{m}}\rightarrow\sqrt{2a+q\sqrt{-7m}}$}\tabularnewline
		&&&\\
		\hline
		&&&\\
		$8$ & $K(\alpha,\beta,\gamma)$ & {\footnotesize$\left(\begin{array}{ccc}\sqrt{\Delta} & -3\alpha+2\beta+\gamma &\alpha\beta-3\beta\gamma+2\gamma\alpha\\ \sqrt{\Delta} &\alpha-3\beta+2\gamma &2\alpha\beta+\beta\gamma-3\gamma\alpha\\ \sqrt{\Delta} & 2\alpha+\beta-3\gamma&-3\alpha\beta+2\beta\gamma+\gamma\alpha \end{array}\right)$}& $-$ \tabularnewline
		&&&\\
		\hline
		&&&\\
		$9$ & $\begin{array}{c} \widetilde{K(\sqrt[7]{\beta_1},\sqrt[7]{\beta_2},\sqrt[7]{\beta_3})}\\  \text{see Lemma \ref{caso9}} \end{array}$ &{\footnotesize $\phi_{0}^{-1}\circ \left(\begin{array}{ccc}\sqrt[7]{\beta_1} & \beta_1\sqrt[7]{\beta_1} &\beta_1^2\sqrt[7]{\beta_1}\\ \sqrt[7]{\beta_2} &\beta_2\sqrt[7]{\beta_2} &\beta_2^2\sqrt[7]{\beta_2}\\ \sqrt[7]{\beta_3} &\beta_3\sqrt[7]{\beta_3}&\beta_3^2\sqrt[7]{\beta_3} \end{array}\right)$} & $\beta_1,\beta_2,\beta_3\rightarrow\beta_1^6,\beta_2^6,\beta_3^6$ \tabularnewline
		&&&\\
		\hline
		&&&\\
		& $k(\zeta_7,\sqrt[7]{m})$ & {\footnotesize $\phi_{0}^{-1}\circ\text{diag}(\sqrt[7]{m},\sqrt[7]{m^4},\sqrt[7]{m^2})$} & $m\rightarrow m^6$\tabularnewline
		&&&\\
		\hline
		&&&\\
		$10$& $\begin{array}{c} K(i,\sqrt[8]{-7\gamma^2})\\ k(\sqrt{2})=k \end{array}$ &{\footnotesize $\phi_{3}^{-1}\circ\text{diag}(\sqrt[8]{-7\gamma^2},\frac{\sqrt[8]{(-7\gamma^2)^7}}{7\gamma},\frac{1}{1+2\sqrt{2}+\sqrt{-7}})$ }& $\gamma\rightarrow \bar{\gamma},\gamma^3,\bar{\gamma}^3$ \tabularnewline
		&&&\\
		\hline
		&&&\\
		$11$ & $\begin{array}{c} K(\alpha_1,...,\alpha_7)\\ \text{see Lemma \ref{caso11}}\end{array}$ & $\begin{array}{c} \mathcal{E}M\mathcal{F} \\ \text{see Proposition \ref{prop11}}\end{array}$& $-$ \tabularnewline
		&&&\\
		\hline
	\end{longtable}
	
\end{theorem}

\begin{remark}
	Some the twists of the Klein quartic are compute in \cite{HK} using a completely different approach.
\end{remark}

\begin{proof}
	The Galois embedding problems for the pairs in cases $1,2,3,4,7,8$ are well-known (see \cite{Gam}) or easy to solve. For cases $5,6,9,10,11$, we use Lemmas \ref{caso5}, \ref{C8}, \ref{caso9}, \ref{caso10} and \ref{caso11}. Then, it is easy to check that the provided isomorphisms satisfy $\xi_{\sigma}=\phi\cdot^{\sigma}\phi^{-1}$, see Proposition \ref{prop11} for the last case.
\end{proof}

\begin{lemma}\label{caso5}	
	The fields $L$ that provide solution to the Galois embedding problem for pair $(G,H)=$ $(<14,1>,<7,1>)$ are given by the only index $3$ Galois extensions of the fields $\tilde{L}=k(\zeta_7,\sqrt[7]{\beta_1},\sqrt[7]{\beta_2},\sqrt[7]{\beta_3})$ with $\beta_1\in k(\zeta_7+\zeta^{6}_7)$ and where $\beta_2$ and $\beta_3$ are its conjugates.
\end{lemma}

\begin{proof}
	Consider a solution with splitting field $L$ and the field extension $\tilde{L}=L(\zeta_7)$. Then $\left[ \tilde{L}:L\right]=3$ and $\tilde{L}=k(\zeta_7,\sqrt[7]{\alpha})$, where $\alpha\in k(\zeta_7)$ by Kummer theory. Since $\operatorname{Gal}(\tilde{L}/k)\simeq <42,4>$, then $\alpha\notin K$. We will prove that we can assume $\alpha\in k(\zeta_7+\zeta^{6}_7)$. Let $\tau\in\operatorname{Gal}(\tilde{L}/k)$ be an element of order $6$, then $\sqrt[7]{\alpha}+\tau^3(\sqrt[7]{\alpha})$ is only fixed by $\tau^3$, but $(\sqrt[7]{\alpha}+\tau^3(\sqrt[7]{\alpha}))^7$ is also fixed by any of the order seven elements. Let us denote by $\sigma$ one of these elements. If we denote $\beta=(\sqrt[7]{\alpha}+\tau^3(\sqrt[7]{\alpha}))^7$, it is $\tilde{L}=k(\zeta_7,\sqrt[7]{\beta})$ and $\beta\in k(\zeta_7+\zeta_{7}^6)$. The field $L$ is the only normal subextension of index $3$. Let us denote $\beta_1=\beta$, and $\beta_2$, $\beta_3$ its two conjugates in $\tilde{L}$. Then $\sqrt[7]{\beta_1\beta_2\beta_3}\in k$. 
\end{proof}

\begin{lemma}\label{C8} Let $L$ be the splitting field for a proper solution to the Galois embedding problem corresponding to pair $(G,H)=(<8,1>,<4,1>)$. Then we have $k(i,\sqrt{2})=k$ and $L=k(\sqrt[8]{-7\gamma^2})$, for some $\gamma\in k$.
\end{lemma}

\begin{proof} Let us define $k_0=k(i,\sqrt{2})$ and $L_0=L(i,\sqrt{2})$. Then, by Kummer theory, $L_0=k_0(\sqrt[8]{-7\gamma^2})$ for some $\gamma\in k_0$. Since $L_0=L\cdot k_0$, we have $\operatorname{Gal}(L_0/k)=\operatorname{Gal}(L/k)\times \operatorname{Gal}(k_0/k)=\langle \sigma_0\rangle\times\langle b\sigma_1,\,\sigma_2\rangle$, where
	$$
	\sigma_0:\,\zeta_8,\sqrt[8]{-7\gamma^{2}}\rightarrow\zeta_{8},\zeta_{8}\sqrt[8]{-7\gamma^{2}},\,
	\sigma_1:\,\zeta_8,\sqrt[8]{-7\gamma^{2}}\rightarrow\zeta_{8}^{7},\sqrt[8]{-7\gamma^{2}_{1}},\,
	\sigma_2:\,\zeta_8,\sqrt[8]{-7\gamma^{2}}\rightarrow\zeta_{8}^{5},\sqrt[8]{-7\gamma^{2}_{2}},\,
	$$
	where $\gamma_1,\gamma_2$ are the corresponding conjugates of $\gamma$ in $k_0$, and $\sigma_0^8=\sigma_1^2=c\sigma_2^2=1$ and they commute. Which implies that $\sigma_0(\sqrt[8]{-7\gamma^{2}_{1}})=\zeta_{8}^{7}\sqrt[8]{-7\gamma^{2}}$.
	
	Then, the element $\sqrt[4]{-7\gamma\gamma_1}$ is fixed by the action of $\sigma_0$ and $\sigma_1$. Therefore $-7\gamma\gamma_1=s^4$ for some $s\in k(\sqrt{2})$. Moreover, $\gamma_1=q^8\gamma^7$ for some $q\in k_0$. As a consequence, $-7q^8\gamma^8=s^4$ and $\sqrt{-7}\in k_0$, which is a contradiction with $\sigma-1\in \operatorname{Gal}(L_0/k)$. Idem with $\sigma_2$. So, finally, $L_0=L$ and $k_0=k=k(i,\sqrt{2})$.
\end{proof}

\begin{lemma}\label{caso9} Let $L$ be the splitting field for a proper solution to the Galois embedding problem corresponding to pair $(G,H)=(<42,1>,<21,1>)$. If $\zeta_7 \in L$, then $L=k(\zeta_7,\sqrt[7]{m})$, for some $m\in k$ (see Proposition $4.1$ in \cite{Lor14}). Otherwise, $L$ is the only normal subextension of index $3$ in $k(\zeta_7,\sqrt[7]{\beta_1},\sqrt[7]{\beta_2},\sqrt[7]{\beta_3})$, where $k(\beta_1,\beta_2,\beta_3)$ is cyclic degree extension of $k$ and $\beta_1,\beta_2,\beta_3$ are conjugated numbers. 
\end{lemma}

\begin{proof}	
	If $\zeta_7 \notin L$, we consider the field $\tilde{L}=L(\zeta_7)$, for which $\left[ \tilde{L}:L\right]=3$. There is a normal subextension of order $3$ over $k$ not contained in $k(\zeta_7)$, which we denote by $F_0$. Then $F=F_0(\zeta_7)$ is a subextenion of index $7$ of $\tilde{L}$. By Kummer theory, $\tilde{L}=k(\zeta_7,\sqrt[7]{\beta})$ with $\beta \in F$. As in Lemma \ref{caso5}, we can assume $\beta \in F_0$. The extension $L/k$ is the only normal subextension of degree $42$ of $\tilde{L}/k$. Denote $\beta_1=\beta$ and $\beta_2$ and $\beta_3$ its two conjugates in $\tilde{L}$. Then $\sqrt[7]{\beta_1\beta_2\beta_3}\in k$. 
\end{proof}

\begin{lemma}\label{caso10} Let $L$ be the splitting field for a proper solution to the Galois embedding problem corresponding to pair $(G,H)=(<16,7>,<8,3>)$. Then $k=k(\sqrt{2})$ and $L=k(\sqrt{m},\sqrt[8]{-7\gamma^2})$, where $\gamma\in k(\sqrt{m})=k(i,\sqrt{2})$.
\end{lemma}
\begin{proof}
	By lemma \ref{C8}, $L=k(\sqrt{m},\sqrt[8]{-7\gamma^2})$, where $\gamma\in k(\sqrt{m})$ and $k(i,\sqrt{2})\subseteq k(\sqrt{m})$. We will check that $k=k(\sqrt{2})$, $k(i,\sqrt{2})=k(\sqrt{m})$ and $\gamma\in k$.
	The group $\operatorname{Gal}(L/k)=\operatorname{Gal}(L/k(\sqrt{m}))\rtimes\operatorname{Gal}(k(\sqrt{m})/k)=\langle \sigma_0,\sigma_1\rangle$ where
	$$
	\sigma_0:\,\zeta_{8},\sqrt{m},\sqrt[8]{-7\gamma^2}\rightarrow\zeta_{8},\sqrt{m},\zeta_8\sqrt[8]{-7\gamma^2},\,\sigma_1:\,\zeta_{8}^\epsilon,\sqrt{m},\sqrt[8]{-7\gamma^2}\rightarrow\zeta_{8},-\sqrt{m},\sqrt[8]{-7\bar{\gamma}^2}.
	$$
	Here, the conditions on $A$ are obtained from the relations $\sigma_0^8=\sigma_1^2=1$ and $\sigma_0\sigma_1=\sigma_1\sigma_0^7$. Now, we can discard cases $(\epsilon_1,\epsilon_2)=(1,1),(1,-1),(-1,-1)$ by proceeding as in the proof of lemma \ref{C8}. Therefore, $\sqrt{2}\in k$, $i\not\in k$, and $L=k(i\sqrt[8]{-7\gamma^2})$ for some $\gamma \in k(i)$. Since $A=1$, we can take $\gamma\in k$.
\end{proof}

To deal with the last and most complicated case we proceed as follows: let us first define the matrix
$$
E=\left(\begin{array}{ccccccc} 1 & \zeta & \zeta^2 & \zeta^3 & \zeta^4 & \zeta^5 & \zeta^6 \\
1 & \zeta^4 & \zeta & \zeta^5 & \zeta^2 & \zeta^6 & \zeta^3 \\
1 & \zeta^2 & \zeta ^4 & \zeta^6 & \zeta & \zeta^3 & \zeta^5
\end{array}\right),
$$
and the matrices $\mathcal{E}=(E\mid\bar{E})$ and $\mathcal{F}=(\frac{F}{\bar{F}})$, where $F=\bar{E}^t$.

We now fix an embedding $\iota:\, <168,42>\hookrightarrow S_7\times S_7$ given by
$
\iota(g)=(1,2,3,4,5,6,7)_1(1,2,3,4,5,6,7)_2,$ $\iota(h)=(2,3,5)_1(4,7,6)_1(2,3,5)_2(4,7,6)_2$ and $\iota(s)=(2,3)_1(4,7)_1(2,3)_2(4,7)_2.$

Then, the following nice property holds: given $\alpha_0\in\text{Aut}(C_K)$, the matrix $\mathcal{E}\alpha_0\mathcal{F}$ is the $14\times14$ matrix given by the permutation $\iota(\alpha_0)$.

\begin{lemma}\label{caso11}
	All the splitting fields to solutions to the Galois embedding problem corresponding to the pair $(G,H)=(<336,208>,<168,42>)$ can be written as $k(\alpha_1,...,\alpha_7,\beta_1,...,\beta_7)$ where $\bar{\alpha}_i=\beta_i$ and the $\alpha_i$ are the $7$ roots of a polynomial of degree $7$ with coefficients in $K=k(\sqrt{-7})$ and such that the action of $\text{Gal}(L/K)$ on the $\alpha_i$ and $\beta_i$ is given by the embedding $\iota:\,<168,42>\hookrightarrow S_7\times S_7$. Complex conjugation acts by switching the $\alpha_i$ and $\beta_i$. 
\end{lemma}

\begin{proof}
	There is a single subgroup up to conjugation of $S_7$ isomorphic to $<168,42>$, so, renaming, if necessary, the roots $\alpha_i$, we get the wanted action of $\text{Gal}(L/K)$.
\end{proof}

Given such a solution, we now define the matrix
$$
\Phi=\left(\begin{array}{ccccccc}1 & \alpha_1 & \alpha_{1}^{2} & \alpha_{1}^{3}& \alpha_{1}^{4}& \alpha_{1}^{5}& \alpha_{1}^{6}\\
1 & \alpha_2 & \alpha_{2}^{2} & \alpha_{2}^{3}& \alpha_{2}^{4}& \alpha_{2}^{5}& \alpha_{2}^{6}\\
1 & \alpha_3 & \alpha_{3}^{2} & \alpha_{3}^{3}& \alpha_{3}^{4}& \alpha_{3}^{5}& \alpha_{3}^{6}\\
1 & \alpha_4 & \alpha_{4}^{2} & \alpha_{4}^{3}& \alpha_{4}^{4}& \alpha_{4}^{5}& \alpha_{4}^{6}\\
1 & \alpha_5 & \alpha_{5}^{2} & \alpha_{5}^{3}& \alpha_{5}^{4}& \alpha_{5}^{5}& \alpha_{5}^{6}\\
1 & \alpha_6 & \alpha_{6}^{2} & \alpha_{6}^{3}& \alpha_{6}^{4}& \alpha_{6}^{5}& \alpha_{6}^{6}\\
1 & \alpha_7 & \alpha_{7}^{2} & \alpha_{7}^{3}& \alpha_{7}^{4}& \alpha_{7}^{5}& \alpha_{7}^{6}\end{array}\right),
$$
and we construct
$$
M:=\left(\begin{array}{c|c} \Phi & \sqrt{-7}\Phi \\
\hline
\bar{\Phi} & -\sqrt{-7}\bar{\Phi}\end{array}\right).
$$

\begin{proposition}\label{prop11}
	The cocycle $\xi_{\sigma}$ is given by the twist $\phi:\,C'\rightarrow C_K$ defined by
	$$
	\phi=\mathcal{E}M\mathcal{F}.
	$$
\end{proposition}

\begin{proof}
	The matrix $M$ previously defined satisfies
	$$
	M\cdot^{\sigma}M^{-1}=\left(\begin{array}{cc}\iota(\xi_{\sigma}) & 0\\0 & \iota(\bar{\xi}_{\sigma})\end{array}\right).
	$$
	Since $\mathcal{E}\alpha_0\mathcal{F}$ is the matrix corresponding to the permutation $\iota(\alpha_0)$ the result follows.
\end{proof}

Now assume $\sqrt{-7}\in k$. We obtain the following possibilities for the pairs $(G,H)$, where $G=H$ because the group $\operatorname{Gal}(K/k)$ is trivial:

\vskip 0.5truecm

\begin{center}
	\begin{tabular}{|c|c|c|c|c|c|c|c|}
		\hline
		& $ \operatorname{ID}( G) $ \rule[0.4cm]{0cm}{0cm} & $\operatorname{gen}(G)$&$n_{(G,G)}$&& $ \operatorname{ID}( G) $ \rule[0.4cm]{0cm}{0cm} & $\operatorname{gen}(G)$&$n_{(G,G)}$ \tabularnewline
		\hline
		$1$ & $<1,1>$ & $ 1 $&$ 1$&
		$9$ & $<21,1>$ &  $ g,\,h $&$2$\tabularnewline
		\hline
		$2$ & $<2,1>$ &  $ s $&$1$&
		$10$ & $<8,3>$ &  $ g^2sg^3sg^2,\,g^2sg^5 $&$2$\tabularnewline
		\hline
		$3$ & $<3,1>$  & $ h $&$1$&
		$11$ & $<168,42>$ &  $ s,\,g,\,h $&$2$\tabularnewline
		\hline
		$4$ & $<7,1>$  & $ g $&$2$&
		$12$ & $<12,3>$ & $ h,\,sg^2sg^5 $&$1$\tabularnewline
		\hline
		$5$ & $<4,2>$  & $hsh^2,\,g^2sg^5$&$1$&
		$13$ & $<12,3>$ & $ h,\,sg^5sg^2 $&$1$\tabularnewline
		\hline
		$6$ & $<4,2>$  &$hsh^2,\,g^5sg^2$ &$1$&
		$14$ & $<24,12>$ & $ s,\,h,\,g^2sg^5 $&$1$\tabularnewline
		\hline
		$7$ & $<4,1>$  & $ g^2sg^3sg^2 $&$1$&
		$15$ & $<24,12>$ & $ s,\,h,\,g^5sg^2 $&$1$\tabularnewline
		\hline
		$8$ & $<6,1>$ &  $ h,\,s $&$1$&\multicolumn{4}{c|}{} \tabularnewline
		\hline
		
	\end{tabular}
\end{center}

\vspace{1.2mm}

All cases except $5,6$ and $12-15$ had already appeared when $\sqrt{-7}\notin k$. We only get $6$ new cases. Since $\sqrt{-7}\in k$, we can work with the model $C_{S_4}:\,x^4+y^4+z^4+3\epsilon(x^2y^2+y^2z^2+z^2x^2)=0$ or the conjugated one $\bar{C}_{S_4}:\,x^4+y^4+z^4+3\bar{\epsilon}(x^2y^2+y^2z^2+z^2x^2)=0$. 

For cases $5$ and $6$, we find the new twists with field of definition $L=k(\sqrt{a},\sqrt{b})$ where $\alpha,\beta\in k^{*}/(k^*)^2$ given by 
$$
\phi=\left(\begin{array}{ccc}\sqrt{a}& 0 & 0\\ 0& \sqrt{b} & 0\\ 0& 0 & 1 \end{array}\right):\,C'\rightarrow C_{S_4}\text{ or }\bar{C}_{S_4}.
$$

For cases $12-15$, we obtain $L=k(\sqrt{\alpha},\sqrt{\beta},\sqrt{\gamma})$, where $\alpha,\beta,\gamma$ are the three roots of a polynomial of degree $3$ with coefficients in $k$, whose splitting field has Galois group over $k$ isomorphic to $\text{C}_3$ or $S_3$ respectively:
$$
\phi=\left(\begin{array}{ccc}\sqrt{\alpha}&\alpha\sqrt{\alpha}&\alpha^2\sqrt{\alpha}\\ \sqrt{\beta}&\beta\sqrt{\beta}&\beta^2\sqrt{\beta}\\ \sqrt{\gamma}&\gamma\sqrt{\gamma}&\gamma^2\sqrt{\gamma}\end{array}\right):\,C'\rightarrow C_{S_4}\text{ or }\bar{C}_{S_4}.
$$

\newpage
\section*{Tables}\label{Tables}

\small

\begin{table}[H]
	\caption{Automorphisms Henn classification}
	\label{aut}
	\begin{center}
		\begin{tabular}{|c|c|c|c|c|c|}
			\hline 
			&&&&&\\
			Case & $\text{Aut}\left(C\right)$ & Generators in $\operatorname{PGL}_3\left(\mathbb{C}\right)$&Case & $\text{Aut}\left(C\right)$ & Generators in $\operatorname{PGL}_3\left(\mathbb{C}\right)$\tabularnewline
			&&&&&\\
			\hline 
			&&&&&\\
			I & $\text{C}_{2}$ & {\footnotesize $\left(\begin{array}{ccc}
				-1 & 0 & 0\\
				0 & 1 & 0\\
				0 & 0 & 1
				\end{array}\right)$}& II & $V_{4}$ & {\footnotesize $\left(\begin{array}{ccc}
				-1 & 0 & 0\\
				0 & 1 & 0\\
				0 & 0 & 1
				\end{array}\right)$, $\left(\begin{array}{ccc}
				1 & 0 & 0\\
				0 & -1 & 0\\
				0 & 0 & 1
				\end{array}\right)$}\tabularnewline
			&&&&&\\
			\hline 
			&&&&&\\
			III & $\text{C}_{3}$ & {\footnotesize $\left(\begin{array}{ccc}
				1 & 0 & 0\\
				0 & 1 & 0\\
				0 & 0 & \zeta_3
				\end{array}\right)$ }& IV & $S_{3}$ & {\footnotesize $\left(\begin{array}{ccc}
				0 & 1 & 0\\
				1 & 0 & 0\\
				0 & 0 & 1
				\end{array}\right)$, $\left(\begin{array}{ccc}
				\zeta_3 & 0 & 0\\
				0 & \zeta_3^{2} & 0\\
				0 & 0 & 1
				\end{array}\right)$}\tabularnewline
			&&&&&\\
			\hline 
			&&&&&\\
			V & $D_{4}$ & {\footnotesize $\left(\begin{array}{ccc}
				0 & 1 & 0\\
				1 & 0 & 0\\
				0 & 0 & 1
				\end{array}\right)$, $\left(\begin{array}{ccc}
				i & 0 & 0\\
				0 & -i & 0\\
				0 & 0 & 1
				\end{array}\right)$}& VI & $\text{C}_{6}$ & {\footnotesize $\left(\begin{array}{ccc}
				-1 & 0 & 0\\
				0 & 1 & 0\\
				0 & 0 & \zeta_3
				\end{array}\right)$}\tabularnewline
			&&&&&\\
			\hline 
			&&\multicolumn{4}{c|}{}\\
			VII & $\text{GAP}\left(16,13\right)$ & \multicolumn{4}{c|}{\footnotesize $\left(\begin{array}{ccc}
				-1 & 0 & 0\\
				0 & 1 & 0\\
				0 & 0 & 1
				\end{array}\right),\,\left(\begin{array}{ccc}
				i & 0 & 0\\
				0 & -i & 0\\
				0 & 0 & 1
				\end{array}\right),\,\left(\begin{array}{ccc}
				0 & 1 & 0\\
				1 & 0 & 0\\
				0 & 0 & 1
				\end{array}\right)$}\tabularnewline
			&&\multicolumn{4}{c|}{}\\
			\hline
			&&&&&\\
			VIII & $S_{4}$ & {\footnotesize $\left(\begin{array}{ccc}
				0 & 0 & 1\\
				1 & 0 & 0\\
				0 & 1 & 0
				\end{array}\right),\,\left(\begin{array}{ccc}
				0 & -1 & 0\\
				1 & 0 & 0\\
				0 & 0 & 1
				\end{array}\right)$}& IX & $\text{C}_{9}$ & {\footnotesize $\left(\begin{array}{ccc}
				\zeta_3 & 0 & 0\\
				0 & 1 & 0\\
				0 & 0 & \zeta_9
				\end{array}\right)$} \tabularnewline
			&&&&&\\
			\hline 
			&&\multicolumn{4}{c|}{}\\
			X & $\operatorname{GAP}\left(48,33\right)$ & \multicolumn{4}{c|}{\footnotesize $\left(\begin{array}{ccc}
				\frac{i\zeta_3^{2}\sqrt{3}}{3} & 0 & \frac{i\zeta_3\sqrt{3}}{3}\\
				0 & \zeta_3^{2} & 0\\
				\frac{2\sqrt{3}\zeta_3^{2}}{3} & 0 & \frac{-i\zeta_3\sqrt{3}}{3}
				\end{array}\right),\,\left(\begin{array}{ccc}
				\frac{\zeta_3\sqrt{3}}{3} & 0 & \frac{\zeta_3^{2}\sqrt{3}}{3}\\
				0 & \zeta_3 & 0\\
				\frac{2\sqrt{3}\zeta_3}{3} & 0 & \frac{-\sqrt{3}\zeta_3^{2}}{3}
				\end{array}\right)$}\tabularnewline
			&&\multicolumn{4}{c|}{}\\
			\hline 
			&&\multicolumn{4}{c|}{}\\
			XI & $\text{GAP}\left(96,64\right)$ & \multicolumn{4}{c|}{\footnotesize $\left(\begin{array}{ccc}
				0 & 0 & 1\\
				1 & 0 & 0\\
				0 & 1 & 0
				\end{array}\right),\,\left(\begin{array}{ccc}
				0 & -i & 0\\
				1 & 0 & 0\\
				0 & 0 & i
				\end{array}\right)$}\tabularnewline
			&&\multicolumn{4}{c|}{}\\
			\hline 
			&&\multicolumn{4}{c|}{}\\
			XII & $\operatorname{PSL}_2\left(\mathbb{F}_7\right)$ & \multicolumn{4}{c|}{\footnotesize $\alpha_1=\frac{-1}{\sqrt{-7}}\left(\begin{array}{ccc}
				\zeta_7-\zeta_7^{6} & \zeta_7^{2}-\zeta_7^{5} & \zeta_7^{4}-\zeta_7^{3}\\
				\zeta_7^{2}-\zeta_7^{5} & \zeta_7^{4}-\zeta_7^{3} & \zeta_7-\zeta_7^{6}\\
				\zeta_7^{4}-\zeta_7^{3} & \zeta_7-\zeta_7^{6} & \zeta_7^{2}-\zeta_7^{5}
				\end{array}\right),\,\alpha_2=\left(\begin{array}{ccc}
				0 & 1 & 0\\
				1 & 0 & 0\\
				0 & 1 & 0
				\end{array}\right),\,\alpha_3=\left(\begin{array}{ccc}
				\zeta_7^{4} & 0 & 0\\
				0 & \zeta_7^{2} & 0\\
				0 & 0 & \zeta_7
				\end{array}\right)$ }\tabularnewline
			&&\multicolumn{4}{c|}{}\\
			\hline 
		\end{tabular}
	\end{center}
\end{table}

\newpage

\begin{table}[H]
	\caption{Automorphisms Modified Henn classification}
	\label{autM}
	\begin{center}
		\begin{tabular}{|c|c|c|}
			\hline 
			Case & $\text{Aut}\left(C\right)$ & Generators in $\operatorname{PGL}_3\left(\mathbb{C}\right)$\tabularnewline
			\hline 
			&&\\
			II & $V_{4}$ & {\footnotesize $\varphi^{-1}\left(\begin{array}{ccc}
				-1 & 0 & 0\\
				0 & 1 & 0\\
				0 & 0 & 1
				\end{array}\right)\varphi$, $\varphi^{-1}\left(\begin{array}{ccc}
				1 & 0 & 0\\
				0 & -1 & 0\\
				0 & 0 & 1
				\end{array}\right)\varphi,\text{ where }\varphi=\left(\begin{array}{ccc}
				1 & \alpha & \alpha^2\\
				1 & \beta & \beta^2\\
				1 & \gamma & \gamma^2
				\end{array}\right)$}\tabularnewline
			&&\\
			\hline 
			&&\\ 
			IV & $S_{3}$ & {\footnotesize $\left(\begin{array}{ccc}
				0 & 1 & 0\\
				1 & 0 & 0\\
				0 & 0 & 1
				\end{array}\right)$, $\left(\begin{array}{ccc}
				\zeta_3 & 0 & 0\\
				0 & \zeta_{3}^{2} & 0\\
				0 & 0 & 1
				\end{array}\right)$}\tabularnewline
			&&\\
			\hline 
			&&\\
			V & $D_{4}$ & {\footnotesize $\left(\begin{array}{ccc}
				0 & \sqrt[4]{a} & 0\\
				1/\sqrt[4]{a} & 0 & 0\\
				0 & 0 & 1
				\end{array}\right)$, $\left(\begin{array}{ccc}
				i & 0 & 0\\
				0 & -i & 0\\
				0 & 0 & 1
				\end{array}\right)$}\tabularnewline
			&&\\
			\hline 
			&&\\
			VI & $\text{C}_{6}$ & {\footnotesize $\left(\begin{array}{ccc}
				-1 & 0 & 0\\
				0 & 1 & 0\\
				0 & 0 & \zeta_3
				\end{array}\right)$}\tabularnewline
			&&\\
			\hline 
			&&\\
			VII & $\text{GAP}\left(16,13\right)$ & {\footnotesize $\left(\begin{array}{ccc}
				-1 & 0 & 0\\
				0 & 1 & 0\\
				0 & 0 & 1
				\end{array}\right)$, $\left(\begin{array}{ccc}
				i & 0 & 0\\
				0 & -i & 0\\
				0 & 0 & 1
				\end{array}\right)$, $\left(\begin{array}{ccc}
				0 & \sqrt[4]{a} & 0\\
				1/\sqrt[4]{a} & 0 & 0\\
				0 & 0 & 1
				\end{array}\right)$}\tabularnewline
			&&\\
			\hline 
		\end{tabular}
	\end{center}
\end{table}

\newpage
\begin{table}[H]
	\caption{Pairs (G,H) for diagonal twists}
	\label{GHdiagonal}
	\begin{center}
		\begin{tabular}{|c|c|c|c|c|c|c|c|c|c|c|c|}
			\hline
			\multicolumn{12}{|c|}{Diagonal twists}\tabularnewline
			\hline
			\hline
			& $ \operatorname{ID}( G) $ \rule[0.4cm]{0cm}{0cm} & $ \operatorname{ID}(H)$ & $\operatorname{gen}(H)$ & $h$& $n_{(G,H)}$& & $ \operatorname{ID}( G) $ \rule[0.4cm]{0cm}{0cm} & $ \operatorname{ID}(H)$ & $\operatorname{gen}(H)$ & $h$& $n_{(G,H)}$\tabularnewline
			\hline
			$1$ & $<2,1>$ & $<1,1>$ & $1$ & $1$& $1$&$8$ & $<8,3>$ & $<4,1>$ & $t^{3}utu$ & $1$& $2$\tabularnewline
			\hline
			$2$ & $<2,1>$ & $<1,1>$ & $1$ & $t^{3}utu$& $1$&$9$ & $<8,3>$ & $<4,1>$ & $t^{3}utu$ & $u$& $2$\tabularnewline
			\hline
			$3$ & $<4,2>$ & $<2,1>$ & $t^{2}$ & $1$& $1$&$10$ & $<8,3>$ & $<4,1>$ & $t^{3}utu^{3}$ & $1$& $2$\tabularnewline
			\hline
			$4$ & $<4,2>$ & $<2,1>$ & $t^{2}$ & $u$& $1$&$11$ & $<8,3>$ & $<4,1>$ & $t^{3}utu^{3}$ & $u$& $2$\tabularnewline
			\hline
			$5$ & $<4,2>$ & $<2,1>$ & $t^{2}$ & $t^{3}utu$& $1$&$12$ & $<16,11>$ & $<8,2>$ & $t^{3}utu,u^{2}$ & $1$& $8$\tabularnewline
			\hline
			$6$ & $<8,5>$ & $<4,2>$ & $t^{2},u^{2}$ & $1$& $1$&$13$ & $<16,11>$ & $<8,2>$ & $t^{3}utu,u^{2}$ & $u$& $16$\tabularnewline
			\hline
			$7$ & $<8,5>$ & $<4,2>$ & $t^{2},u^{2}$ & $u$& $3$&$14$ & $<32,34>$ & $<16,2>$ & $u,t^{3}ut$ & $1$& $64$\tabularnewline
			\hline
		\end{tabular}
	\end{center}
\end{table}

\vspace{-3mm}

\begin{table}[H]
	\caption{Pairs (G,H) for almost-diagonal twists}
	\label{GHalmost}
	\begin{center}
		\begin{tabular}{|c|c|c|c|c|c|c|c|c|c|c|c|}
			\hline
			\multicolumn{12}{|c|}{Almost-diagonal twists}\tabularnewline
			\hline
			\hline
			& $ \operatorname{ID}( G) $ \rule[0.4cm]{0cm}{0cm} & $ \operatorname{ID}(H)$ & $\operatorname{gen}(H)$ & $h$& $n_{(G,H)}$& & $ \operatorname{ID}( G) $ \rule[0.4cm]{0cm}{0cm} & $ \operatorname{ID}(H)$ & $\operatorname{gen}(H)$ & $h$& $n_{(G,H)}$\tabularnewline
			\hline
			$1$ & $<2,1>$ & $<1,1>$ & $1$ & $u^{2}t$& $1$&
			$19$ & $<16,6>$ & $<8,2>$ & $t^{3}utu,u^{2}$ & $ut$& $2$\tabularnewline
			\hline
			$2$ & $<4,2>$ & $<2,1>$ & $t^{2}$ & $u^{2}t$& $1$&
			$20$ & $<16,13>$ & $<8,2>$ & $t^{3}utu,u^{2}$ & $u^{2}t$& $2$\tabularnewline
			\hline
			$3$ & $<4,1>$ & $<2,1>$ & $t^{2}$ & $t$& $1$&
			$21$ & $<16,8>$ & $<8,1>$ & $u^{2}tu$ & $u$& $2$\tabularnewline
			\hline
			$4$ & $<4,2>$ & $<2,1>$ & $u^{2}t$ & $1$& $2$&
			$22$ & $<16,7>$ & $<8,3>$ & $t^{3}utu^{3},u^{2}t$ & $u$& $2$\tabularnewline
			\hline
			$5$ & $<4,2>$ & $<2,1>$ & $u^{2}t$ & $t^{3}utu$& $2$&
			$23$ & $<16,8>$ & $<8,4>$ & $t^{3}utu^{3},t$ & $u$& $2$\tabularnewline
			\hline
			$6$ & $<8,4>$ & $<4,1>$ & $t^{3}utu$ & $t$& $2$&
			$24$ & $<16,11>$ & $<8,3>$ & $t,u^{2}$ & $1$& $2$\tabularnewline
			\hline
			$7$ & $<8,1>$ & $<4,1>$ & $t^{3}utu^{3}$ & $tu$& $1$&
			$25$ & $<16,13>$ & $<8,4>$ & $t^{3}utu^{3},t$ & $1$& $6$\tabularnewline
			\hline
			$8$ & $<8,3>$ & $<4,2>$ & $t^{2},u^{2}$ & $u^{2}t$& $1$&
			$26$ & $<16,13>$ & $<8,2>$ & $t,utu$ & $1$& $4$\tabularnewline
			\hline
			$9$ & $<8,2>$ & $<4,1>$ & $t$ & $1$& $1$&
			$27$ & $<16,11>$ & $<8,2>$ & $t,utu$ & $u^{2}$& $8$\tabularnewline
			\hline
			$10$ & $<8,5>$ & $<4,2>$ & $u^{2}t, t^{2}$ & $1$& $6$&
			$28$ & $<16,11>$ & $<8,3>$ & $t,u^{2}$ & $utu$& $2$\tabularnewline
			\hline
			$11$ & $<8,2>$ & $<4,1>$ & $t^{3}utu^{3}$ & $u^{2}t$& $1$&
			$29$ & $<16,7>$ & $<8,1>$ & $tu^{2}tut$ & $tutu^{2}$& $4$\tabularnewline
			\hline
			$12$ & $<8,3>$ & $<4,1>$ & $t ^{3}utu$ & $u^{2}t$& $4$&
			$30$ & $<16,11>$ & $<8,3>$ & $t^{3}utu^{3},u^{2}t$ & $1$& $2$\tabularnewline
			\hline
			$13$ & $<8,5>$ & $<4,2>$ & $u^{2}t,t^{2}$ & $t^{3}utu$& $6$&
			$31$ & $<32,43>$ & $<16,6>$ & $tu^{2}tut,u^{2}$ & $u$& $4$\tabularnewline
			\hline
			$14$ & $<8,3>$ & $<4,1>$ & $t$ & $t^{3}utu^{3}$& $2$&
			$32$ & $<32,11>$ & $<16,2>$ & $u,t^{3}ut$ & $u^{2}t$& $4$\tabularnewline
			\hline
			$15$ & $<8,2>$ & $<4,1>$ & $t$ & $t^{3}utu$& $1$&
			$33$ & $<32,43>$ & $<16,13>$ & $t^{3}utu,u^{2},t$ & $u$& $4$\tabularnewline
			\hline
			$16$ & $<8,3>$ & $<4,2>$ & $u^{2}t,t^{2}$ & $u^{2}$& $1$&
			$34$ & $<32,7>$ & $<16,6>$ & $tu^{2}tut,u^{2}$ & $1$& $4$\tabularnewline
			\hline
			$17$ & $<8,3>$ & $<4,2>$ & $u^{2}t,t^{2}$ & $t^{3}utu^{3}$& $1$&
			$35$ & $<32,49>$ & $<16,13>$ & $utu,u^{2},t$ & $1$& $12$\tabularnewline
			\hline
			$18$ & $<8,3>$ & $<4,1>$ & $t$ & $u^{2}$& $2$&
			$36$ & $<64,134>$ & $<32,11>$ & $t,u$ & $1$& $8$\tabularnewline
			\hline
		\end{tabular}
	\end{center}
\end{table}

\vspace{-3mm}

\begin{table}[H]
	\caption{Pairs (G,H) for non-diagonal twists}
	\label{GHnondiag}
	\begin{center}
		\begin{tabular}{|c|c|c|c|c|c|c|c|c|c|c|c|}
			\hline
			\multicolumn{12}{|c|}{Non-diagonal twists}\tabularnewline
			\hline
			\hline
			& $ \operatorname{ID}( G) $ \rule[0.4cm]{0cm}{0cm} & $ \operatorname{ID}(H)$ & $\operatorname{gen}(H)$ & $h$& $n_{(G,H)}$& & $ \operatorname{ID}( G) $ \rule[0.4cm]{0cm}{0cm} & $ \operatorname{ID}(H)$ & $\operatorname{gen}(H)$ & $h$& $n_{(G,H)}$\tabularnewline
			\hline
			$1$ & $<6,1>$ & $<3,1>$ & $s$ & $u^{2}t$&$1$&
			$6$ & $<48,48>$ & $<24,12>$ & $s,t$ & $1$&$1$\tabularnewline
			\hline
			$2$ & $<6,2>$ & $<3,1>$ & $s$ & $1$&$1$&
			$7$ & $<96,64>$ & $<48,3>$ & $s,u,t^{3}ut$ & $u^{2}t$&$2$\tabularnewline
			\hline
			$3$ & $<12,4>$ & $<6,1>$ & $s,u^{2}t$ & $1$&$1$&
			$8$ & $<96,72>$ & $<48,3>$ & $s,u,t^{3}ut$ & $1$&$4$\tabularnewline
			\hline
			$4$ & $<24,12>$ & $<12,3>$ & $s,u^{2}$ & $u^{2}t$&$1$&$9$ & $<192,956>$ & $<96,64>$ & $s,t,u$ & $1$&$2$\tabularnewline
			\hline
			$5$ & $<24,13>$ & $<12,3>$ & $s,u^{2}$ & $1$&$1$&\multicolumn{6}{|c|}{}\tabularnewline
			\hline
		\end{tabular}
	\end{center}
\end{table}

In the tables above the "$\text{gen}(H)$" and "$h$" columns serve to reconstruct $G$ and $H$.
They show generators of $H$, and a single matrix $h$. The meaning is that $ G$ is the group which elements are $(g,1)$ for $g$ in $H$ together with the elements $(gh,\tau)$ for $g$ in $H$ and $\tau$ the non-trivial automorphism in $\operatorname{Gal}(K/k)$. The number $n_{(G,H)}$ is the one in Remark \ref{R3} and introduced in \cite{Lor14}.

\end{document}